\documentclass{article}

\usepackage{floatrow}
\usepackage{graphicx,subcaption}
\usepackage{amsmath, amsfonts, amsthm, amssymb, bm}
\usepackage{color}
\usepackage{tabularx,hhline}
\usepackage{multirow}
\usepackage{booktabs}
\usepackage{fullpage}[2cm]
\usepackage{algorithm}
\usepackage{algorithmic}
\usepackage{enumerate}
\usepackage{paralist}
\usepackage{cancel}
\usepackage{tikz}
\usepackage{xspace}
\usetikzlibrary{patterns}
\usepackage{listings}
\usepackage[normalem]{ulem}

\usepackage[margin = 1.5cm, font = small, labelfont = bf, labelsep = period]{caption}
\usepackage[font = small, labelfont = bf, labelsep = period]{caption}
\newfloatcommand{capbtabbox}{table}[][\FBwidth]

\newcommand{\subjectto}{\textnormal{subject to}}
\newcommand{\st}{\textnormal{s.t.}}

\newcommand{\cone}{\textnormal{cone}}
\newcommand{\vol}{\textnormal{Vol}}

\newcommand{\minimize}{\textnormal{minimize}}
\newcommand{\interior}{\textnormal{int}}
\newcommand{\norm}[1]{\left\lVert#1\right\rVert}

\DeclareMathOperator*{\Diag}{Diag}

\newcommand{\eg}{e.g.}
\newcommand{\ie}{i.e.}
\newcommand{\PP}{\mathbb{P}}
\newcommand{\QQ}{\mathbb{Q}}
\newcommand{\EE}{\mathbb{E}}
\newcommand{\RR}{\mathbb{R}}

\newcommand{\Sbb}{\mathbb{S}}
\newcommand{\Pc}{\mathcal{P}}
\newcommand{\Qc}{\mathcal{Q}}
\newcommand{\Ib}{\mathbb{I}}

\newcommand{\gap}{\vspace{1mm}}

\newcommand{\Emve}{\E_{\textnormal{mve}}}
\newcommand{\Esmvie}{\E_{\textnormal{smvie}}}
\newcommand{\Ecop}{\E_{\textnormal{sdp}}}
\newcommand{\Esproc}{\E_{\textnormal{sproc}}}
\newcommand{\Ektt}{\E_{\textnormal{ktt}}}
\newcommand{\Ezrh}{\E_{\textnormal{zrh}}}

\newcommand{\sproc}{$\mathcal S$-procedure\xspace}

\newcommand{\zab}{Z(\bm A, \bm b)}

\newcommand{\Sb}{{\bm{\mathcal S}}}
\newcommand{\C}{\mathcal{C}}
\newcommand{\X}{\mathcal{X}}

\newcommand{\E}{\mathcal{E}}
\newcommand{\R}{\mathcal{R}}
\newcommand{\K}{\mathcal{K}}

\newcommand{\ee}{\mathbf e}

\newcommand{\etal}{et al.}

\newcommand{\tr}{\textup{tr}}

\newtheorem{theorem}{Theorem}
\newtheorem{proposition}{Proposition}
\newtheorem{lemma}{Lemma}

\newtheorem{corollary}{Corollary}
\newtheorem{example}{Example}
\newtheorem{remark}{Remark}

\newtheorem{assume}{Assumption}

\begin{document}

%
%
%
%

\title{Finding Minimum Volume Circumscribing Ellipsoids Using Generalized Copositive Programming}
\author{Areesh Mittal and Grani A.~Hanasusanto \\ Graduate Program in Operations Research and Industrial Engineering\\ The University of Texas at Austin, USA}
\date{ }
\maketitle

\begin{abstract}
We study the problem of finding the L\"{o}wner-John ellipsoid, \ie, an ellipsoid with minimum volume that contains a given convex set. We reformulate the problem as a generalized copositive program, and use that reformulation to derive tractable semidefinite programming approximations for instances where the set is defined by affine and quadratic inequalities. We prove that, when the underlying set is a polytope, our method never provides an ellipsoid of higher volume than the one obtained by scaling the maximum volume inscribed ellipsoid.  We empirically demonstrate that our proposed method generates high-quality solutions faster than solving the problem to optimality. Furthermore, we outperform the existing approximation schemes in terms of solution time and quality. 
We present applications of our method to obtain piecewise-linear decision rule approximations for dynamic distributionally robust problems with random recourse, and to generate ellipsoidal approximations for the set of reachable states in a linear dynamical system when the set of allowed controls is a polytope.
\end{abstract}

\section{Introduction}

We consider the \emph{minimum volume ellipsoid problem} (MVEP), which can be stated as follows  \cite{boyd2004convex, todd2016minimum}: ``Given a set $\Pc \subset \RR^K$, find an ellipsoid $\Emve$ with minimum volume that contains $\Pc$.'' In this paper, we focus on sets $\Pc$ that satisfy the following assumption.
\begin{assume}\label{assum:set}
	The set $\Pc$ is compact, convex, and full-dimensional.
\end{assume}
\noindent Compactness guarantees the existence of a bounding ellipsoid. The convexity assumption is made without loss of generality; if the set is not convex, then we can instead consider its convex hull without affecting $\Emve$. If $\Pc$ is not full-dimensional, then the ellipsoid $\Emve$ is degenerate with zero volume. For sets $\Pc$ satisfying Assumption \ref{assum:set}, such an ellipsoid, also known as the \emph{L\"{o}wner-John ellipsoid}, is unique and affine-invariant, making it an attractive outer approximation of $\Pc$ \cite[Section 8.4.1]{boyd2004convex}. The MVEP arises in many applications studied in the literature. Several authors discuss outer ellipsoidal approximations for the set of reachable points in control systems \cite{kurzhanskiui1997ellipsoidal, calafiore2004ellipsoidal}, as it is easier to check whether a point lies in an ellipsoid than in the comparatively complicated reachable set. Rimon and Boyd \cite{rimon1997obstacle} advocate the use of $\Emve$ for collision detection in robotics. Here, one checks whether the ellipsoids intersect as opposed to the sets that they approximate. Other applications of the MVEP include outlier detection \cite{ahipacsaouglu2015fast,silverman1980minimum}, pattern recognition \cite{glineur1998pattern}, computer graphics \cite{eberly20013d}, and facility location \cite{elzinga1974minimum}.  We refer the reader to \cite{henk2012lowner} for an excellent article about the lives of the eponymous researchers Karel L\"{o}wner and Fritz John, the history of the MVEP which dates back to late 1930s, and some important properties of L\"{o}wner-John ellipsoids.

For some sets $\Pc$, it is possible to identify $\Emve$ in polynomial time. For example, if $\Pc$ is defined as the convex hull of a finite number of points, then the complexity of finding $\Emve$ is polynomial in the problem size \cite{khachiyan1996rounding,sun2004computation}. When $\Pc$ is a union of ellipsoids, one can employ the $\mathcal S$-lemma to compute $\Emve$ in polynomial time~\cite{yildirim2006minimum}. However, excluding these special cases, finding $\Emve$ is, in general, a difficult problem. For example, if $\Pc$ is a polytope defined by affine inequalities, or if $\Pc$ is an intersection of ellipsoids, then finding $\Emve$ is NP-hard \cite{boyd2004convex,freund1985complexity}.

Gotoh and Konno \cite{gotoh2006minimal} present a constraint-generation approach to solve the MVEP exactly when $\Pc$ is a polytope defined by affine inequalities. The method starts with a collection of points contained in $\Pc$ and finds the ellipsoid of minimum volume containing those points. Then feasible points lying outside the current ellipsoid are successively generated, and the ellipsoid is updated to include the new point, until a desired optimality tolerance is reached. However, generating a point that lies in $\Pc$ but outside the current candidate ellipsoid at each iteration is very slow, as it entails solving a non-convex optimization problem. Therefore, this approach is computationally expensive, and one has to resort to approximation schemes.

One popular approximation method for the MVEP is based on identifying and scaling the \emph{maximum volume inscribed ellipsoid} (MVIE), \ie, the ellipsoid with maximum volume contained in $\Pc$. In particular, it is known that scaling the MVIE around its center by a factor of $K$ results in an ellipsoid that contains~$\Pc$, thereby serving as an approximation of $\Emve$~\cite{john2014extremum}. Moreover, the MVIE can be found in polynomial time if $\Pc$ is defined by affine and quadratic inequalities \cite{khachiyan1993complexity}. However, this technique, which we refer to as the SMVIE (scaled MVIE) approach, produces highly suboptimal ellipsoids because of the scaling factor~$K$. Another method for approximating the MVEP utilizes the well-known \sproc. Boyd \etal\ \cite{boyd1994linear} discuss the application of the \sproc to generate approximations for the MVEP when $\Pc$ is either an intersection or a Minkowski sum of ellipsoids. {Finally, in a recent paper, Zhen \etal \ \cite{zhen2017robust} study approximations to uncertain second order cone programs and demonstrate how this framework can be exploited to derive an approximation to the MVEP}. 

Several authors have identified sufficient conditions under which a convex set contains another convex set. 
 Helton \etal\ \cite{helton2013matricial} discuss sufficient conditions which guarantee that a \emph{semidefinite-representable} set contains another such set. Kellner \etal\ \cite{kellner2013containment} provide slightly improved sufficient conditions compared to the ones in \cite{helton2013matricial}. Although these articles do not focus on the MVEP specifically, their results can be used to approximate $\Emve$ if $\Pc$ is semidefinite-representable (see Appendix \ref{appendix:ktt}). To the best of our knowledge, there are no results that provide a finite system of constraints that are necessary and sufficient to ensure that an ellipsoid contains another set. This gap in knowledge is our main focus.

In this article, we prove that checking whether an ellipsoid contains $\Pc$ is \emph{equivalent} to solving a finite-dimensional \emph{generalized copositive} (GC) feasibility problem. We use this result to reformulate the MVEP exactly as a GC program. This representation of the MVEP enables us to leverage state-of-the-art approximation schemes available for GC programming problems. In particular, GC programs yield a hierarchy of optimization problems which provide an increasingly tight restriction to the original problem \cite{lasserre2001global,parrilo2000structured,zuluaga2006lmi}. While our exact reformulation holds for any $\Pc$ satisfying Assumption \ref{assum:set}, we focus primarily on developing approximations in the case where $\Pc$ is defined by affine and convex quadratic inequalities. We demonstrate that, for these sets, the resulting approximation can be formulated as a semidefinite program (SDP), which can be solved in polynomial time. Since these SDPs are restrictions of the GC reformulation, they provide a feasible ellipsoid that contains $\Pc$. There has been previous work on developing exact copositive programming reformulations for otherwise difficult problems, and using those reformulations to generate tractable approximations \cite{bomze2012copositive, burer2012copositive, burer2012representing, mittal2019robust, hanasusanto2018conic, NRZ11:mixed01,prasad2018improved}. Our results add to this literature by demonstrating the ability of generalized copositive programs to exactly model the MVEP.


We demonstrate the utility of our approximations to the MVEP in two applications. First, we consider a two-stage distributionally robust optimization (DRO) problem with \emph{random recourse}. Such a problem is NP-hard even in the absence of random recourse \cite{bertsimas2010models}. Bertsimas and Dunning \cite{bertsimas2016multistage} study a \emph{piecewise static} decision rules approximation for the case of dynamic robust optimization, which leads to a tractable reformulation. Although they do not consider a DRO model, this approach can be extended to such a setting. In contrast, we focus on piecewise \emph{linear} decision (PLD) rules approximation. In the presence of random recourse, finding the optimal PLD rule is NP-hard, although feasible PLD rules can be obtained using the \sproc. Unfortunately, these decision rules are often of poor quality. The effectiveness of the \sproc in finding good PLD rules can be improved by considering an ellipsoid that contains the \emph{support set}, \ie, the set of allowed values for the uncertain parameters. In the context of an inventory management problem, we show that the size of this ellipsoid can have a large effect on the quality of the PLD rules. We also demonstrate that the PLD rules generated using our method significantly outperform the piecewise static decision rules.  Second, we utilize our method to generate high-quality ellipsoidal approximations to the set of reachable states in a linear dynamical system when the \emph{control set}, \ie, the set of allowed controls, is a polytope.
This complements the existing schemes that provide similar approximations when the control set is an ellipsoid \cite{kurzhanskiui1997ellipsoidal, calafiore2004ellipsoidal}.

We summarize the main contributions of the article below.
\begin{enumerate}[1)]
	\item We provide necessary and sufficient finite-dimensional conic inequalities that certify whether an ellipsoid contains a set $\Pc$ satisfying Assumption \ref{assum:set}. We use these conditions to derive a generalized copositive reformulation of the MVEP. 
	\item  When $\Pc$ is defined by affine and convex quadratic inequalities, we derive a tractable SDP restriction to the GC reformulation, which results in a feasible ellipsoid that contains $\Pc$. We prove that the volume of the resulting ellipsoid never exceeds that of the SMVIE approach. To the best of our knowledge, our approximation is the first one to have this property. We further show that the ratio of the volume of the SMVIE to the volume of the ellipsoid generated by our method can be arbitrarily high. {We also prove that both the \sproc \cite[Section~3.7]{boyd1994linear} and the approximation suggested by Zhen \etal\ \cite{zhen2017robust} never generates ellipsoids of lower volumes than the SMVIE approach.}
	\item We demonstrate through extensive numerical experiments that our method is significantly faster than solving the problem to optimality using the constraint-generation technique of \cite{gotoh2006minimal}. The experiments further indicate that our method significantly outperforms the SMVIE approach in terms of solution quality. Also, our method outperforms the scheme which utilizes the sufficient conditions of Kellner \etal~\cite{kellner2013containment} both in terms of solution time and quality.
	\item We present two important applications of our approach. Firstly, we exploit the bounding ellipsoids to obtain improved decision rule approximations to two-stage DRO models with random recourse, which have resisted effective solution schemes so far. Secondly, we provide ellipsoidal approximations for the set of reachable states in a linear dynamical system when the control set is a polytope.
\end{enumerate}

This article is organized as follows. In Section \ref{sec:copos_refor}, we describe the MVEP and reformulate it as an equivalent GC program. In Section \ref{sec:approx_polytope}, we use that reformulation to derive a tractable SDP that generates a near-optimal approximation to $\Emve$ when the set is defined by affine and quadratic inequalities. In Section~\ref{sec:dro}, we explain the application of our approach for obtaining improved decision rules approximation for a two-stage DRO model with random recourse. {In Section \ref{sec:experiments}, we present numerical experiments comparing the volumes of the ellipsoids generated by our method against those found using other approaches. We also demonstrate the efficacy of our approach in solving a distributionally robust inventory management problem. Finally, we conclude in Section \ref{sec:conc}. Auxiliary proofs and additional numerical experiments can be found in the electronic companion to the paper.}

\subsection{Preliminaries}
\label{sec:prelims}

\paragraph{Notation} For a positive integer $I$, we use $[I]$ to denote the index set $\{ 1,2,\ldots, I \}$. We denote the vector of ones by $\ee$, and the identity matrix by $\Ib$; their dimensions will be clear from the context. We use $\RR^K(\RR_+^K)$ to denote the set of (non-negative) vectors of length $K$, and $\Sbb^K(\Sbb_+^K)$ to denote the set of all $K\times K$ symmetric (positive semidefinite) matrices. In addition, $\Sbb_{++}^K$ represents the set of positive definite matrices. The functions $\tr(\cdot)$ and $\det(\cdot)$ denote the trace and the determinant of the input matrix, respectively. We define $\Diag(\bm v)$ as a diagonal matrix with vector $\bm v$ on its main diagonal. The symbols $\norm{\bm v}_1$ and $\norm{\bm v}$ denote the $\ell_1$-norm and $\ell_2$-norm of vector $\bm v$, respectively. The vertical concatenation of two scalars or vectors $\bm u$ and $\bm v$ is denoted by $[\bm u; \bm v]$. For a matrix $\bm M \in \RR^{I\times J}$, we use $\bm M_{:j} \in \RR^I$ to denote its $j$-th column, and $\bm M_{i:}\in \RR^J$ to denote the transpose of its $i$-th row. We represent the interior and the conic hull of a set $S$ by $\interior(S)$ and $\cone(S)$ respectively.

\paragraph{Generalized Copositive Matrices} We use $\C(\K)$ to denote the set of generalized copositive matrices with respect to cone $\K\subseteq \RR^K$, \ie, $\C(\K) = \{ \bm M \in \Sbb^K : \bm x^\top \bm M \bm x \geq 0\ \forall \bm x \in \K \}$. The set of copositive matrices is a special case of such a set when $\K = \RR_+^K$. We use $\C^\ast(\K)$ to denote the set of generalized completely positive matrices with respect to cone $\K$, \ie, $\C^\ast(\K) = \{ \bm M \in \Sbb^K: \bm M = \sum_{i \in [I]} \bm x_i \bm x_i^\top , \bm x_i \in \K \}$ where $I$ is a positive integer. The cones $\C(\K)$ and $\C^\ast(\K)$ are duals of each other \cite{sturm2003cones}. For any $\bm P, \bm Q\in\Sbb^K$ and cone $\bar{\C} \subseteq \Sbb^K$, the conic inequality $\bm P\succeq_{\bar{\C}}\bm Q$ indicates that $\bm P-\bm Q$ is an element of $\bar{\C}$. We drop the subscript and simply write $\bm P \succeq \bm Q$, when $\bar{\C} = \Sbb^K_+$. Finally, the relation $\bm M \succ_{\C(\K)}\bm 0$ indicates that $\bm M$ is strictly copositive, \ie, $\bm x^\top \bm M \bm x > 0$ for all $\bm x \in \K, \bm x \neq \bm 0$.

\paragraph{Ellipsoids} We define $\E(\bm A, \bm b) = \{ \bm x \in \RR^K:\ \norm{\bm A \bm x + \bm b}^2 \leq 1 \}$ as an ellipsoid with parameters $\bm A\in \Sbb_{++}^K$ and $\bm b \in \RR^K$. The volume of $\E(\bm A,\bm b)$, denoted by $\vol(\E(\bm A,\bm b))$, is proportional to $\det(\bm A^{-1}) = 1/\det(\bm A)$. In this paper, we drop the proportionality constant, and say that $\vol(\E(\bm A,\bm b)) = 1/\det(\bm A)$; since we use the volume as a metric for comparing different ellipsoids, doing so does not affect the results. We define the \emph{radius} of a $K$-dimensional ellipsoid as $\vol(\cdot)^{1/K}$; this metric is proportional to the radius of a sphere with the same volume as that of the ellipsoid. Finally, we say the two ellipsoids are equal, \ie, $\E(\bm A_1,\bm b_1) = \E(\bm A_2,\bm b_2)$, if and only if $\bm A_1 = \bm A_2$ and $\bm b_1 = \bm b_2$.

\section{Generalized Copositive Reformulation}
\label{sec:copos_refor}
In this section, we develop a generalized copositive reformulation for the MVEP. It is well-known that $\Emve = \E(\bm A,\bm b)$ if and only if $(\bm A,\bm b)$ is the unique optimal solution to the following semi-infinite convex optimization problem \cite{todd2016minimum}:
\begin{equation}
\label{eq:MVE}
\begin{array}{clll}
\displaystyle\minimize&\displaystyle -\log\det(\bm A)\\
\subjectto& \bm A\in \Sbb^K,\; \bm b \in \RR^K,\; \zab \leq 1,
\end{array}
\tag{$\mathcal{MVE}$}
\end{equation}
where
\begin{equation}
\label{eq:MVE_cons}
\zab = \sup_{\bm x \in \Pc}\norm{\bm A\bm x + \bm b}^2 = \sup_{\bm x \in \Pc}\left\{\bm x^\top \bm A^2 \bm x + 2\bm b^\top \bm A \bm x + \bm b^\top \bm b\right\}.
\end{equation}
The objective function of \eqref{eq:MVE} minimizes the logarithm of the volume, which implicitly restricts $\bm A$ to be positive definite. Minimizing the logarithm of the volume makes the objective function convex in $\bm A$. The constraint $\zab\leq 1$ forces every element of $\Pc$ to lie inside the ellipsoid. We are now ready to present the main result of this section, where we derive necessary and sufficient conditions for certifying whether an ellipsoid contains another set.



\begin{theorem}
	\label{thm:copos_refor}
	Let $\Pc$ be a set satisfying Assumption \ref{assum:set}. Let the cone $\K \subseteq \RR^{K+1}$ be defined as
	\begin{equation}
	\label{eq:def_cone}
	\K = \cone\left(\{[\bm x; 1]: \ \bm x \in \Pc\}\right).
	\end{equation}
	If $\bm A\in \Sbb_{++}^K$ and $\bm b\in \RR^K$, then the ellipsoid $\E(\bm A,\bm b)$ contains $\Pc$ if and only if there exist $\bm F\in \Sbb^K,\; \bm g\in \RR^K,\; h\in \RR,$ such that
	\begin{equation}
	\label{eq:gc_conditions}
	\begin{array}{clll}
	\begin{bmatrix} \bm F  & \bm g \\ \bm g^\top  & h - 1 \end{bmatrix} \preceq_{\C(\K)} \bm 0 \quad \text{and} \quad \begin{bmatrix} \bm F  & \bm g & \bm A\\ \bm g^\top & h& \bm b^\top\\ 
	\bm A & \bm b& \bm \Ib \end{bmatrix} \succeq \bm 0.
	\end{array}
	\end{equation}
\end{theorem}

Before proving Theorem \ref{thm:copos_refor}, we discuss its implications.  The theorem implies that the constraint ${\zab\leq 1}$ in \eqref{eq:MVE} can be replaced by the constraints in \eqref{eq:gc_conditions}. Therefore, $\Emve = \E(\bm A,\bm b)$ is the minimum volume ellipsoid if and only if $(\bm A, \bm b , \bm F,\bm g, h)$ is the unique optimal solution to the following generalized copositive program:
	\begin{equation}
	\label{eq:MVE_COP}
	\begin{array}{clll}
	\minimize &-\log\det(\bm A)\\
	\subjectto&\bm A\in \Sbb^K,\; \bm b \in \RR^K,\; \bm F\in \Sbb^K,\; \bm g\in \RR^K,\; h\in \RR,\\
	&\eqref{eq:gc_conditions} \textnormal{ holds}.
	\end{array}
	\end{equation}
\begin{remark}
	In this article, we refer to a problem with $-\log\det(\cdot)$ minimization objective and semidefinite (copositive) constraints as a ``semidefinite (copositive) program,'' albeit with a slight abuse of terminology. The reason is that minimization of $-\log\det(\cdot)$ is equivalent to minimization of $-(\det(\cdot))^{1/K}$; the latter can be reformulated as a problem with linear objective and additional semidefinite inequality constraints (see, \eg, \cite[Section 4.2]{ben2001lectures}). Some modeling frameworks, like YALMIP \cite{yalmip} which we use for our experiments, automatically carry out this conversion before sending the problem to the solver. 
\end{remark}
Next, we present the following technical lemmas which are needed for the proof of Theorem \ref{thm:copos_refor}.
\begin{lemma}
	\label{lem:tau0}
	Let $\K$ be the cone defined in \eqref{eq:def_cone}. If $[\bm x; \tau] \in \K$, then $\tau \geq 0$. Furthermore, $\tau = 0$ only if $\bm x = \bm 0$.
\end{lemma}
\begin{proof}
	From the definition of $\K$, there exist points $\bm x_s \in \Pc$ and coefficients \mbox{$\lambda_s\geq 0$}, $s \in [S]$, such that \mbox{$[\bm x;\tau] = \sum_{s \in [S]}\lambda_s[\bm x_s;1]$.} By comparing the last element, we get $\tau= \sum_{s \in [S]} \lambda_s \geq 0 $, since $\lambda_s\geq 0$. In addition, $\tau = 0$ implies that $\lambda_s = 0$ for all~$s \in [S]$, which further implies that $\bm x = \bm 0$.
\end{proof}
{
	
	\begin{lemma}
		\label{lem:proper_cone}
		The cone $\K$ defined in \eqref{eq:def_cone} is proper.
	\end{lemma}
	\begin{proof}
		The compactness of $\Pc$ implies that $\K$ is convex and closed. Since $\Pc$ has nonempty interior, any point $\bm x$ in the interior of $\Pc$ yields a point $[\bm x; 1]$ in the interior of $\K$; therefore $\K$ has nonempty interior. Finally to see that $\K$ is pointed, let $[\bm x; \tau] \in \K $ and  $-[\bm x; \tau] \in \K$. Using Lemma \ref{lem:tau0}, we have that $\tau \geq 0$ and $-\tau \geq 0$, which implies that $\tau = 0$. Again using Lemma \ref{lem:tau0}, we get that $[\bm x; \tau] = \bm 0$, which implies that $\K$ is pointed.  
	\end{proof}
	
	\begin{lemma}
		\label{lem:quadratic_constant}
		Let $\bm M \in \Sbb^K$ be a symmetric matrix and $S \subseteq \RR^K$ be a set with nonempty interior. If $\bm v^\top \bm M \bm v = 0$ for all $\bm v \in S$, then $\bm M = \bm 0$.
	\end{lemma}
	\begin{proof}
		Let $\lambda$ be an eigenvalue of $\bm M$ and $\bm q$ be the corresponding eigenvector of unit length. Since $S$ has nonempty interior, for any $\bm v \in \interior(S)$, there exists $\overline{\tau} > 0$ such that $\bm v + \tau \bm q \in S$ for all $\tau \in [0, \overline{\tau}]$. Therefore, $(\bm v + \tau \bm q)^\top \bm M \bm (\bm v + \tau \bm q) = 0$ for all $\tau \in [0, \overline{\tau}].$ Furthermore,
		\begin{equation*}
		(\bm v + \tau \bm q)^\top \bm M \bm (\bm v + \tau \bm q) = \bm v^\top \bm M \bm v + 2\tau \bm q^\top \bm M \bm v + \tau^2 \bm q^\top\bm M \bm q = 2\tau \lambda \bm q^\top \bm v + \tau^2 \lambda \bm q^\top \bm q = \lambda \tau (2\bm q^\top \bm v + \tau).
		\end{equation*}
		Thus, $ \lambda \tau (2\bm q^\top \bm v + \tau) = 0$ for all $\tau \in [0, \overline{\tau}]$. Since the term $\tau(2\bm q^\top \bm v + \tau)$ is quadratic in the scalar $\tau$, it cannot be zero for more than two values of $\tau$. This implies that the previous equality holds for all $\tau \in [0, \overline\tau]$ only if $\lambda = 0$. Therefore, any eigenvalue of $\bm M$ is zero, which implies that $\bm M = \bm 0$.
	\end{proof}

	\begin{lemma}
		\label{lem:slater_point}
		Let $\K$ be the cone defined in \eqref{eq:def_cone}. There exist $\bm X \in \Sbb^K$ and $\bm x \in \RR^K$ such that
		$$\begin{bmatrix}  \bm X & \bm x\\
		\bm x^\top & 1
		\end{bmatrix}\succ_{\C^\ast(\K)} \bm 0.$$
	\end{lemma}
	\begin{proof}
		We start by showing that $\C(\K)$ is pointed. Let $\bm M \in \Sbb^{K+1}$ be such that $\bm M \in \C(\K)$ and $-\bm M \in \C(\K)$. For this choice of $\bm M$, for all $\bm x \in \K$, we have that $\bm x^\top \bm M \bm x \geq 0$ and $-\bm x^\top \bm M \bm x \geq 0$, which implies that $\bm x^\top \bm M \bm x = 0$ for all $\bm x \in \K$. Since $\K$ has non-empty interior (by Lemma \ref{lem:proper_cone}), Lemma \ref{lem:quadratic_constant} implies that $\bm M = \bm 0$. Therefore $\C(\K)$ is pointed, which implies that its dual cone, $\C^\ast(\K)$, has non-empty interior~\cite[Section 2.6.1]{boyd2004convex}. Consider $\bm M \in \interior(\C^\ast(\K))$. The matrix $\bm M$ is positive definite (see discussion below Corollary 8.1 in \cite{burer2012copositive}); therefore any element on its diagonal, which includes the bottom-right component, is strictly positive. By scaling $\bm M$ such that the bottom-right component is $1$, we get another matrix in the interior of $\C^\ast(\K)$. Hence the lemma holds.
	\end{proof}
}

The following lemma is an extension of another recently proved result found in \cite[Lemma 4]{mittal2019robust}.
\begin{lemma}
	\label{lem:linearize_A'A}
	Let $\bm M\in\Sbb^K$ be a symmetric matrix  and $\bm A\in\RR^{J\times K}$ be an arbitrary matrix. Then, for any proper cone $\K\subset \RR^K$, the inequality $ \bm M \succeq_{\C(\K)} \bm A^\top\bm A$ is satisfied if and only if there exists a matrix $\bm H\in\Sbb_+^K$ such that 
	\begin{equation}
	\label{eq:inequalities_for_H}
	\bm M\succeq_{\C(\K)}\bm H\quad\text{ and }\quad \begin{bmatrix} \bm H & \bm A^\top \\ \bm A & \Ib \end{bmatrix}\succeq\bm 0. 
	\end{equation}	
\end{lemma}

\begin{proof}
	($\Rightarrow$) The statement holds immediately by setting $\bm H=\bm A^\top\bm A$.\\
	($\Leftarrow$) Assume that there exists such a matrix $\bm H\in\Sbb_+^K$. By the Schur complement, the second inequality in \eqref{eq:inequalities_for_H} implies that $\bm H\succeq\bm A^\top\bm A$, which in turn implies that $\bm H\succeq_{\C(\K)}\bm A^\top\bm A$ (since $\Sbb_+^K \subseteq \C(\K)$ for any $\K$). Combining this with the first inequality in \eqref{eq:inequalities_for_H} implies that $\bm M \succeq_{\C(\K)} \bm A^\top\bm A$.	
\end{proof}
\noindent We now return to the proof of Theorem~\ref{thm:copos_refor}.
\begin{proof}[Proof of Theorem~\ref{thm:copos_refor}.]
The set $\Pc$ can be expressed in terms of the cone $\K$ as $\Pc = \{\bm x\in \RR^K\ : \ [\bm x; 1] \in \K \}.$ Therefore, we can write \eqref{eq:MVE_cons} as
\begin{equation}
\label{eq:cone_form}
\zab = \sup_{[\bm x ; 1] \in \K}\ \bm x^\top \bm A^2 \bm x + 2\bm b^\top \bm A \bm x + \bm b^\top \bm b.
\end{equation}
The optimization problem \eqref{eq:cone_form} is equivalent to the following completely positive program \cite{burer2012copositive}:
\begin{equation}
\label{eq:CPP}
\begin{array}{clll}
\zab = &\sup&\displaystyle\tr(\bm A^2\bm X)+2\bm b^\top\bm A \bm x + \bm b^\top \bm b\\
&\st& \displaystyle \bm x\in\RR^K,\; \bm X\in \Sbb^K,\\
&&\begin{bmatrix}  \bm X & \bm x\\
\bm x^\top & 1
\end{bmatrix}\succeq_{\C^\ast(\K)} \bm 0.
\end{array}
\end{equation}
The dual of this completely positive program can be written as:
\begin{equation}
\label{eq:COP}
\begin{array}{rlll}
Z_{\text{d}}(\bm A,\bm b) = \displaystyle \inf_{\rho \in \RR} & \rho \\
\st& \begin{bmatrix} - \bm A^2  & - \bm A \bm b \\- \bm b^\top\bm A & \rho - \bm b^\top \bm b \end{bmatrix} \succeq_{\C(\K)} \bm 0.
\end{array}
\end{equation}
{Using Lemma \ref{lem:slater_point}, we conclude that a Slater point exists in the optimization problem~\eqref{eq:CPP}. Hence, strong duality holds and $\zab = Z_{\text{d}}(\bm A,\bm b)$. Furthermore, there exists a dual feasible solution which attains the value $Z(\bm A, \bm b)$, since a Slater point exists in the primal problem \eqref{eq:CPP} \cite[Theorem 1.4.2]{ben2001lectures}}. Using these facts, we have that $\zab \leq 1$ if and only if there exists a feasible solution to problem \eqref{eq:COP} whose objective function value is at most $1$. Therefore, $\zab \leq 1$ if and only if there exists $\rho \leq 1$ such that
\begin{equation*}
\begin{bmatrix} - \bm A^2  & - \bm A \bm b \\ - \bm b^\top\bm A & \rho - \bm b^\top \bm b \end{bmatrix} \succeq_{\C(\K)} \bm 0,
\end{equation*}
which, in turn, holds if and only if
$$\begin{bmatrix} - \bm A^2  & - \bm A \bm b \\ - \bm b^\top\bm A & 1 - \bm b^\top \bm b \end{bmatrix} \succeq_{\C(\K)} \bm 0,$$
or equivalently,
\begin{equation}\label{eq:copos_refor}\begin{bmatrix} \bm 0  & \bm 0 \\ \bm 0 & 1 \end{bmatrix} \succeq_{\C(\K)} \begin{bmatrix} \bm A  & \bm b \end{bmatrix}^\top\begin{bmatrix} \bm A  & \bm b \end{bmatrix}. \end{equation}
The conic inequality \eqref{eq:copos_refor} has non-linearity because of the terms involving the product of the decision variables $\bm A$ and $\bm b$. However, by Lemma \ref{lem:linearize_A'A}, this constraint is satisfied if and only if there exist variables $\bm F \in \Sbb^K,\; \bm g \in \RR^K$ and $h \in \RR$ such that the constraints \eqref{eq:gc_conditions} hold. Therefore, the constraint $\zab\leq 1$ is equivalent to constraints \eqref{eq:gc_conditions}. Hence, the claim follows.
\end{proof}

\noindent Theorem \ref{thm:copos_refor} implies that $\Emve$ can be found by solving the GC program \eqref{eq:MVE_COP}, which is difficult in general. In the next section, we discuss tractable approximations of \eqref{eq:MVE_COP} for special cases of $\Pc$. However, before doing so, we provide some generalizations to the GC reformulation \eqref{eq:MVE_COP}.

{
\begin{remark}[Affine mapping of a set]
	\label{remark:affine_mapping}
	Let $\Pc\subseteq \RR^K$ be a set satisfying Assumption \ref{assum:set}. Let $\overline{\Pc} = \bm C \Pc + \bm d = \{ \bm C \bm x + \bm d : \bm x \in \Pc \} \subset \RR^J$ be an affine mapping of $\Pc$, where $\bm C\in \RR^{J\times K}$ and $\bm d \in \RR^J$. In order to obtain conditions for an ellipsoid to contain $\overline{\Pc}$, note that $Z(\bm A, \bm b) = \sup_{\bm x \in \Pc }\norm{\bm A (\bm C \bm x + \bm d) + \bm b}^2.$ Following the steps of the proof of Theorem \ref{thm:copos_refor}, we can see that if $\bm A\in \Sbb_{++}^J$ and $\bm b\in \RR^J$, then the ellipsoid $\E(\bm A,\bm b)$ contains $\overline{\Pc}$ if and only if there exist $\bm F\in \Sbb^{K},\; \bm g\in \RR^{K},\; h\in \RR,$ such that
	\begin{equation*}
	\begin{array}{clll}
	\begin{bmatrix} \bm F  & \bm g \\ \bm g^\top  & h - 1 \end{bmatrix} \preceq_{\C(\K)} \bm 0 \quad \text{and} \quad \begin{bmatrix} \bm F  & \bm g & (\bm A \bm C)^\top \\ \bm g^\top & h& (\bm A \bm d + \bm b)^\top\\ 
	\bm A \bm C & \bm A \bm d + \bm b& \bm \Ib \end{bmatrix} \succeq \bm 0,
	\end{array}
	\end{equation*}
	where $ \K = \cone\left(\{[\bm x; 1]: \ \bm x \in \Pc\}\right).$
\end{remark}
}
\begin{remark}[Union of sets]
	Let $\Pc = \cup_{\ell \in [L]}\Pc_\ell$, where the set $\Pc_\ell$ satisfies Assumption \ref{assum:set} for all $\ell \in [L]$. The set $\Pc$ does not satisfy Assumption \ref{assum:set} since it may not be convex. However, it is possible to extend Theorem \ref{thm:copos_refor} to this case as follows. Note that an ellipsoid contains the union of sets if and only if it contains every set. We can  apply Theorem \ref{thm:copos_refor} to every set $\Pc_\ell$ to arrive at the fact that ellipsoid $\E(\bm A,\bm b)$ contains $\Pc$ if and only if there exist $\bm F_\ell\in \Sbb^K,\; \bm g_\ell\in \RR^K,\; h_\ell\in \RR\ \forall \ell \in [L]$, such that
	\begin{equation*}
	\begin{array}{clll}
	&\begin{bmatrix} \bm F_\ell  & \bm g_\ell \\ \bm g_\ell^\top  & h_\ell - 1 \end{bmatrix} \preceq_{\C(\K_\ell)} \bm 0 \quad\text{and}\quad \begin{bmatrix} \bm F_\ell  & \bm g_\ell & \bm A\\ \bm g_\ell^\top & h_\ell& \bm b^\top\\ \bm A & \bm b& \bm \Ib \end{bmatrix} \succeq \bm 0 \quad \forall \ell \in [L],
	\end{array}
	\end{equation*}
	where $\K_\ell = \cone\left(\{[\bm x; 1]: \ \bm x \in \Pc_\ell\}\right), \ell \in [L]$.
\end{remark} 

\begin{remark}[Minkowski sum of sets]
	\label{remark:mink_sum}
	For all $\ell \in [L]$, let the set $\Pc_\ell$ satisfy Assumption \ref{assum:set} and $\K_\ell$ be the corresponding cone defined as in \eqref{eq:def_cone}. Let $\Pc = \left\{ \sum_{\ell \in [L]} \bm x_\ell :\ \bm x_\ell \in \Pc_\ell\ \forall \ell \in [L] \right\}$ be the Minkowski sum of these sets.  Although $\Pc$ satisfies Assumption \ref{assum:set}, it might not have a polynomial sized representation. As an example, if every $\Pc_\ell$ is a polytope, then $\Pc$ is a polytope defined by constraints whose number can potentially grow exponentially with $L$. However, we can still reformulate \eqref{eq:MVE} for $\Pc$ as a GC program of polynomial size as follows. Observe that
	\begin{equation*}
	\begin{array}{clll}
	\zab& = \displaystyle\sup_{\bm x_\ell \in \Pc_\ell \forall \ell \in [L]}\left\{ \left(\sum_{\ell \in [L]} \bm x_\ell\right)^\top \bm A^2 \left(\sum_{\ell \in [L]} \bm x_\ell\right) + 2\bm b^\top \bm A \left(\sum_{\ell \in [L]} \bm x_\ell\right) + \bm b^\top \bm b\right\}\vspace{1mm}\\
	& \displaystyle = \sup_{\substack{\bm x =[\bm x_1;\bm x_2;\cdots;\bm x_L],}\atop \substack{\bm x_\ell \in \Pc_\ell\ \forall \ell \in [L]}} \left\{ \bm x^\top \tilde{\bm A} \tilde{\bm A}^\top \bm x + 2\bm b^\top \tilde{\bm A}^\top \bm x + \bm b^\top \bm b\right\},
	\end{array}
	\end{equation*}
	where $\tilde{\bm A} = \begin{bmatrix}\bm A &  \bm A& \cdots& \bm A
	\end{bmatrix}^\top \in \RR^{LK\times K}.$
	By defining the cone $\K$ as $$\K = \{ [\bm x_1;\bm x_2;\cdots;\bm x_L;\tau] \in \RR^{LK+1} : [\bm x_\ell;\tau] \in \K_\ell \; \forall \ell \in [L]  \}$$ and repeating the steps in the proof of Theorem \ref{thm:copos_refor}, we arrive at the fact that ellipsoid $\E(\bm A,\bm b)$ contains $\Pc$ if and only if there exist $ \bm F\in \Sbb^{LK},\; \bm g\in \RR^{LK},\; h\in \RR$ such that
	\begin{equation*}
	\begin{array}{clll}
	&\begin{bmatrix} \bm F  & \bm g \\ \bm g^\top  & h - 1 \end{bmatrix} \preceq_{\C(\K)} \bm 0\quad\text{and}\quad \begin{bmatrix} \bm F  & \bm g & \tilde{\bm A}\\ \bm g^\top & h& \bm b^\top\\ \tilde{\bm A}^\top & \bm b& \bm \Ib \end{bmatrix} \succeq \bm 0.
	\end{array}
	\end{equation*}
\end{remark}
\noindent In the previous three remarks, minimizing the function $-\log\det(\bm A)$ subject to the corresponding constraints leads to a GC reformulation of \eqref{eq:MVE}.


\section{Tractable Approximations for Polytopes}
\label{sec:approx_polytope}

In this section, we use the reformulation \eqref{eq:MVE_COP} to present tractable semidefinite programming approximations for \eqref{eq:MVE} in the case where the set $\Pc$ is a polytope defined as
\begin{equation}\label{eq:define_polytope}\Pc = \left\{\bm x\in \RR^K\ : \ \bm S\bm x \leq \bm t\right\},\end{equation}
where $\bm S\in \RR^{J\times K}$ and $\bm t \in \RR^J$. We start with our proposed approximation, and then present theoretical comparisons with alternative approaches to approximate $\Emve$.

\begin{theorem}
	\label{thm:C0_polytope}
	Let $\Pc$ be a polytope defined as in \eqref{eq:define_polytope} that satisfies Assumption \ref{assum:set}. Consider any $\bm A \in \Sbb_{++}^K$ and $\bm b\in \RR^K$. Then, an ellipsoid $\E(\bm A, \bm b)$ contains $\Pc$ if there exist $\bm N \in \RR_+^{J\times J},\; \bm F\in \Sbb^K,\; \bm g\in \RR^K,\; h\in \RR$ such that
	\begin{equation}
	\label{eq:suff_conds_polytope}
	\begin{array}{clll}
	& \displaystyle \begin{bmatrix} \bm F&\bm g\\\bm g^\top & h - 1 \end{bmatrix} \preceq   -\begin{bmatrix} -\bm S^\top \\ \bm t^\top \end{bmatrix}\bm N\begin{bmatrix} -\bm S &\bm t \end{bmatrix}, \quad \text{and} \quad \begin{bmatrix}\bm F &\bm g & \bm A\\ \bm g^\top & h& \bm b^\top \\ \bm A& \bm b & \Ib \end{bmatrix} \succeq \bm 0.
	\end{array}
	\end{equation}
\end{theorem}
\begin{proof}
	For the polytope $\Pc$, the cone $\K$ defined in \eqref{eq:def_cone} can be written as
	$\K = \left\{ [\bm x; \tau] \in \RR^{K+1}: \; \tau \geq 0,\;\bm S\bm x \leq \tau \bm t \right\}.$
	We show that the constraints \eqref{eq:suff_conds_polytope} imply the constraints \eqref{eq:gc_conditions}. Since the second constraints in \eqref{eq:suff_conds_polytope} and \eqref{eq:gc_conditions} are the same, we show that the first constraint of \eqref{eq:suff_conds_polytope} implies the generalized copositive constraint in \eqref{eq:gc_conditions} which proves our claim. For any $[\bm x;\tau] \in \K$, we have that
	$$\begin{bmatrix} \bm x\\ \tau \end{bmatrix}^\top \begin{bmatrix} \bm F&\bm g\\\bm g^\top & h - 1 \end{bmatrix} \begin{bmatrix} \bm x\\ \tau \end{bmatrix} \leq -\begin{bmatrix} \bm x\\ \tau \end{bmatrix}^\top \begin{bmatrix} -\bm S^\top \\ \bm t^\top \end{bmatrix}\bm N\begin{bmatrix} -\bm S &\bm t \end{bmatrix}\begin{bmatrix} \bm x\\ \tau \end{bmatrix} = - (\tau \bm t - \bm S \bm x)^\top \bm N(\tau \bm t - \bm S \bm x)\leq 0,$$
	where the first inequality follows from the first semidefinite inequality in \eqref{eq:suff_conds_polytope} and the final inequality holds since $\bm N \geq \bm 0$ and $\tau \bm t - \bm S \bm x \geq \bm 0$. Thus, $$\begin{bmatrix} \bm F&\bm g\\\bm g^\top & h - 1 \end{bmatrix} \preceq_{\C(\K)}\bm 0.$$
	Hence, the claim follows.
\end{proof}

Theorem \ref{thm:C0_polytope} provides a way to approximate \eqref{eq:MVE}. We choose the ellipsoid with minimum volume among those that satisfy the conditions of Theorem \ref{thm:C0_polytope}. This can be achieved by solving the following tractable SDP:
\begin{equation}
\label{eq:MVE_approx_polytope}
\begin{array}{clll}
\minimize &-\log\det(\bm A)\\
\subjectto&\bm A\in \Sbb^K,\; \bm b \in \RR^K,\; \bm N \in \RR_+^{J\times J},\; \bm F\in \Sbb^K,\; \bm g\in \RR^K,\; h\in \RR,\\
& \eqref{eq:suff_conds_polytope} \textnormal{ holds.}
\end{array}
\end{equation}
If $(\bm A, \bm b ,\bm N, \bm F,\bm g, h)$ is an optimal solution to \eqref{eq:MVE_approx_polytope}, then we propose the use of the ellipsoid $\Ecop = \E(\bm A, \bm b)$ as an approximation of $\Emve$.


Next, we present a theoretical comparison of the quality of $\Ecop$ with the other methods of approximating $\Emve$. We denote by $\Esmvie$ the ellipsoid obtained by scaling the maximum volume inscribed ellipsoid by a factor of $K$. We discuss the SDP formulation for determining $\Esmvie$ in Appendix \ref{appendix:smvie}. In Theorem \ref{thm:compare_polytope}, presented below, we show that the volume of $\Ecop$ cannot exceed the volume of $\Esmvie$. For the theoretical analysis, it is convenient to combine the two semidefinite inequalities of \eqref{eq:suff_conds_polytope} using the Schur complement, and write \eqref{eq:MVE_approx_polytope} equivalently as follows:
\begin{equation}
\label{eq:MVE_Ab}
\begin{array}{clll}
\minimize &-\log\det(\bm A)\\
\subjectto&\bm A\in \Sbb^K,\; \bm b \in \RR^K,\;\bm N \in \RR_+^{J\times J},\vspace{1mm}\\
& \displaystyle \begin{bmatrix} \bm A \\\bm b^\top \end{bmatrix}\begin{bmatrix} \bm A &\bm b \end{bmatrix} \preceq  \begin{bmatrix} \bm 0 & \bm 0\\ \bm 0 & 1 \end{bmatrix} - \begin{bmatrix} -\bm S^\top \\ \bm t^\top \end{bmatrix}\bm N\begin{bmatrix} -\bm S &\bm t \end{bmatrix}. 
\end{array}
\end{equation}
We begin with the following lemma which we use for comparing the volumes of $\Ecop$ and $\Esmvie$.
\begin{lemma}[{\cite[Theorem 7.8.7]{horn1990matrix}}]
	\label{lem:det_transpose}
	If $\bm M \in \RR^{K\times K}$ is a square matrix with real entries such that \mbox{$\bm M + \bm M^\top \succ \bm 0$}, then  
	$$\det\left(\frac{1}{2}\left(\bm M+ \bm M^\top \right) \right) \leq \det(\bm M).$$
\end{lemma}

\begin{theorem}
	\label{thm:compare_polytope}
	If $\Pc$ is a polytope defined as in \eqref{eq:define_polytope}, then $\vol(\Ecop) \leq \vol(\Esmvie)$.
\end{theorem}
\begin{proof}
The logarithm of the volume of $\Esmvie$ is equal to the optimal value of the following optimization problem (see Appendix \ref{appendix:smvie}): 
\begin{equation}
\label{eq:smvie_dual_proof}
\begin{array}{cll}
\displaystyle \minimize&\displaystyle K\bm \rho^\top \bm t - K - \log\det\left( - \frac{1}{2}\left(\bm S^\top\bm \Lambda  + \bm \Lambda^\top \bm S\right)\right)\\
\subjectto& \bm \Lambda \in \RR^{J\times K},\; \bm \rho \in \RR^J,\\ 
&\bm S^\top \bm \rho = \bm 0,\\
& \norm{\bm \Lambda_{j:}} \leq \rho_j \ \forall j \in [J].
\end{array}
\end{equation}
We can compare the volumes of $\Ecop$ and $\Esmvie$ by comparing the optimal values of the minimization problems \eqref{eq:MVE_Ab} and \eqref{eq:smvie_dual_proof}. To prove the theorem, we show that any feasible solution in \eqref{eq:smvie_dual_proof} can be used to construct a feasible solution to \eqref{eq:MVE_Ab} with the same or lower objective function value. To this end, consider a solution $(\bm \Lambda, \bm \rho)$ which satisfies the constraints of \eqref{eq:smvie_dual_proof}. Define $\kappa = \exp(1 - \bm\rho^\top \bm t)$. Also, let $\bm\Lambda^\top \bm S = \bm U \bm \Sigma \bm V^\top$ be the singular value decomposition of $\bm \Lambda^\top \bm S$, where $\bm U, \bm V\in \RR^{K\times K}$ are orthonormal matrices, and $\bm \Sigma \in \Sbb^{K}$ is a diagonal matrix. We note for later use that $\det(\bm\Lambda^\top\bm S) = \det(\bm U)\det(\bm \Sigma)\det(\bm V^\top)  = \det(\bm \Sigma)$. Consider the following solution to \eqref{eq:MVE_Ab}:
\begin{equation} \label{eq:feasible_soln}\bm A = \kappa \bm V\bm\Sigma \bm V^\top,\; \bm b = \kappa \bm V\bm U^\top \bm \Lambda^\top \bm t,\; \bm N = \kappa^2\left(\bm \rho\bm \rho^\top - \bm \Lambda\bm \Lambda^\top \right).\end{equation}
We demonstrate that this solution satisfies the constraints of \eqref{eq:MVE_Ab}. Note that
$$\bm A^2 = \kappa^2 \bm V\bm\Sigma \bm V^\top\bm V\bm\Sigma \bm V^\top = \kappa^2 \bm V\bm\Sigma \bm U^\top\bm U\bm\Sigma \bm V^\top = \kappa^2 \bm S^\top\bm \Lambda\bm \Lambda^\top \bm S,$$
since $\bm V^\top\bm V = \bm U^\top\bm U = \Ib$. Similarly $\bm A \bm b = \kappa^2 \bm S^\top \bm \Lambda\bm \Lambda^\top \bm t$, and $\bm b^\top \bm b =  \kappa^2\bm t^\top \bm \Lambda\bm \Lambda^\top \bm t$. Therefore,
\begin{equation*}
\begin{array}{rll}
\begin{bmatrix} -\bm S^\top \\ \bm t^\top \end{bmatrix} \bm N\begin{bmatrix} -\bm S &\bm t \end{bmatrix} = \begin{bmatrix} \bm S^\top \bm N \bm S & -\bm S^\top \bm N \bm t \\ -\bm t^\top \bm N\bm S & \bm t^\top \bm N \bm t \end{bmatrix} & = \kappa^2 \begin{bmatrix} \bm S^\top \left(\bm \rho\bm \rho^\top - \bm \Lambda\bm \Lambda^\top \right) \bm S & -\bm S^\top \left(\bm \rho\bm \rho^\top - \bm \Lambda\bm \Lambda^\top \right) \bm t \\ -\bm t^\top \left(\bm \rho\bm \rho^\top - \bm \Lambda\bm \Lambda^\top \right)\bm S & \bm t^\top \left(\bm \rho\bm \rho^\top - \bm \Lambda\bm \Lambda^\top \right) \bm t
\end{bmatrix} \vspace{1mm}\\
&= \kappa^2 \begin{bmatrix} -\bm S^\top\bm \Lambda\bm \Lambda^\top \bm S & \bm S^\top \bm \Lambda\bm \Lambda^\top \bm t \\ \bm t^\top \bm \Lambda\bm \Lambda^\top \bm S & (\bm \rho^\top \bm t)^2 - \bm t^\top \bm \Lambda\bm \Lambda^\top \bm t \end{bmatrix} \vspace{1mm}\\
& = \begin{bmatrix} \bm 0 & \bm 0\\ \bm 0 & (\kappa \bm \rho^\top\bm t)^2 \end{bmatrix} - \begin{bmatrix} \bm A^2 & \bm A\bm b\\ \bm b^\top \bm A & \bm b^\top \bm b \end{bmatrix},
\end{array}
\end{equation*}
where the third equality follows from the constraint $\bm S^\top \bm \rho = \bm 0$. We claim that $(\kappa \bm \rho^\top\bm t)^2 \leq 1$. To see this, first note that since the polytope $\Pc$ is non-empty, by Farkas' Lemma, any vector $\bm \rho$ satisfying $\bm S^\top \bm \rho = \bm 0$ and $\bm \rho \geq \bm 0$ also satisfies $\bm \rho^\top \bm t \geq 0$. Secondly, using the inequality $\exp(\nu) \geq 1 + \nu$ with $\nu = \bm \rho^\top \bm t - 1$, we get that $\kappa^{-1} = \exp(\bm \rho^\top \bm t - 1) \geq \bm\rho^\top \bm t$, which implies that $\kappa \bm \rho^\top \bm t \leq 1$. Combining these two inequalities, we get that $0\leq \kappa \bm \rho^\top \bm t \leq 1$, which implies that $(\kappa \bm \rho^\top\bm t)^2 \leq 1$. Therefore, we have that
$$\begin{bmatrix}\bm A \\ \bm b^\top \end{bmatrix} \begin{bmatrix}\bm A & \bm b \end{bmatrix}  = \begin{bmatrix} \bm 0 & \bm 0\\ \bm 0 & (\kappa \bm \rho^\top\bm t)^2 \end{bmatrix} - \begin{bmatrix} -\bm S \\ \bm t \end{bmatrix}\bm N\begin{bmatrix} -\bm S &\bm t \end{bmatrix}\preceq \begin{bmatrix} \bm 0 & \bm 0\\ \bm 0 & 1 \end{bmatrix} - \begin{bmatrix} -\bm S \\ \bm t \end{bmatrix}\bm N\begin{bmatrix} -\bm S &\bm t \end{bmatrix}.$$
Next, since $\bm N = \kappa^2\left(\bm \rho\bm \rho^\top - \bm \Lambda\bm \Lambda^\top \right)$, we have that $N_{ij} = \kappa^2\left(\rho_i\rho_j - \bm \Lambda_{i:}^\top \bm\Lambda_{j:}\right)\geq \kappa^2\left(\rho_i\rho_j - \norm{\bm \Lambda_{i:}}\norm{ \bm\Lambda_{j:}}\right) \geq 0,$
where the two inequalities follow from Cauchy-Schwarz and the constraint $\norm{\bm \Lambda_{j:}} \leq \rho_j$ respectively. Therefore, $\bm N \geq \bm 0$. Next, we compare the objective values. Note that
\begin{equation*}
\begin{array}{rll}-\log\det(\bm A) = -\log\det(\kappa \bm V\bm\Sigma \bm V^\top) &= -\log(\kappa^K \det(\bm V\bm\Sigma \bm V^\top))\\
& = - K\log(\kappa) - \log(\det(\bm V)\det(\bm \Sigma)\det(\bm V^\top)\\
& = K(\bm \rho^\top\bm t - 1) - \log\det(\bm \Lambda^\top \bm S)\\
&\displaystyle \leq K(\bm\rho^\top \bm t - 1) - \log\det\left(\frac{1}{2}(\bm \Lambda^\top \bm S + \bm S^\top \bm \Lambda )\right),
\end{array}
\end{equation*}
where the final inequality follows from Lemma \ref{lem:det_transpose}. Hence, the feasible solution \eqref{eq:feasible_soln} gives a lower objective function value. Thus, the claim follows.
\end{proof}

\begin{corollary}
	If the polytope $\Pc$ is a simplex, then $\Emve = \Ecop = \Esmvie$.
\end{corollary}
\begin{proof}
	It is known that $\Emve = \Esmvie$, if the set $\Pc$ is a simplex \cite[Section 8.4.1]{boyd2004convex}. Therefore, $\vol(\Esmvie) = \vol(\Emve)$, which implies that $\vol(\Ecop) = \vol(\Emve)$. Because of the uniqueness of the minimum volume ellipsoid, we get that $\Ecop =  \Emve$.
\end{proof}

In the next example, we demonstrate that the difference between the volumes of the ellipsoids $\Ecop$ and $\Esmvie$ can be arbitrarily large.

\begin{example}[Chipped Hypercube]
	\label{example:chipped}
 	Consider the polytope: $\Pc = \{ \bm x\in \RR^K : \bm 0 \leq \bm x \leq \ee, \ee^\top \bm x \leq \sqrt{K} \}$ formed by adding one constraint to the unit hypercube. This polytope forms a special case of \eqref{eq:define_polytope} with
 	$\bm S = \left[\Ib;\; -\Ib;\; \ee^\top\right]$, and $\bm t = \left[ \ee;\; \bm 0;\; \sqrt{K}\right].$ Let $R_{\textnormal{mve}}$, $R_{\textnormal{smvie}}$ and $R_{\textnormal{cop}}$ be the radii (defined in Section \ref{sec:prelims}) of the ellipsoids $\Emve$, $\Esmvie$ and $\Ecop$, respectively. In the e-companion, we prove that  $R_{\textnormal{cop}} = O\left( K^{1/4} \right)$ and $R_{\textnormal{smvie}} = \Theta\left( K^{1/2} \right)$. Therefore,  $R_{\textnormal{smvie}}$ grows at a strictly faster rate with the dimension $K$ than $R_{\textnormal{cop}}$. This example demonstrates that the ratio $R_{\textnormal{smvie}}/ R_{\textnormal{cop}}$ can be arbitrarily high, if a large enough $K$ is chosen. We compute the three radii for $K = 2$ to $K=50$, and plot the values in Figure \ref{fig:chippedsqr}(b). We observe that $R_{\textnormal{mve}} $ is very close to $R_{\textnormal{cop}}$, and the two appear to be growing at the same rate with $K$. Figure \ref{fig:chippedsqr}(a) shows the ellipsoids generated by the three methods for $K=2$.
 	\begin{figure}
 		\begin{center}
		\begin{subfigure}[b]{.45\textwidth}
			\includegraphics[width=1\linewidth]{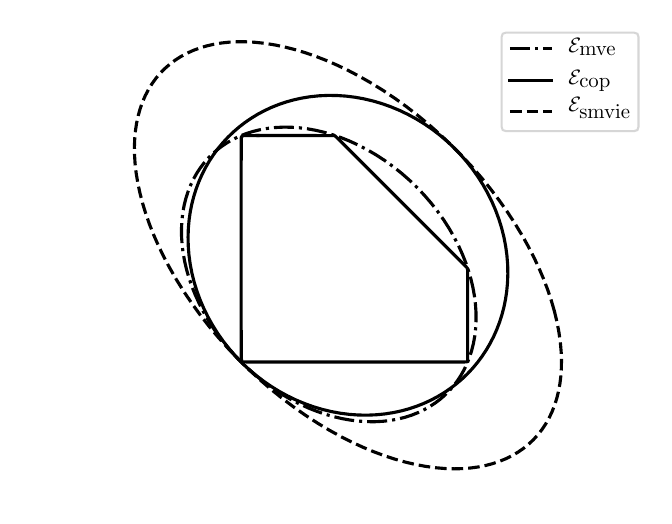}
			\caption{}
		\end{subfigure}\hfill%
		\begin{subfigure}[b]{.45\textwidth}
			\includegraphics[width=1\linewidth]{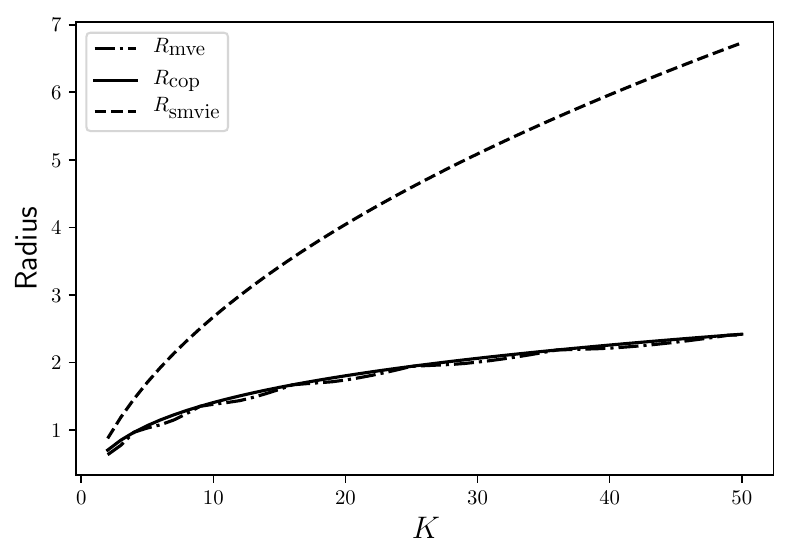}
			\caption{}
		\end{subfigure}\hfill%
		\vspace{-3mm}
		\caption{Chipped Hypercube Example: \textit{a)} The ellipsoids generated by the exact method and the two approximation methods for $K = 2$. \textit{b)} Radii $(\ie, \vol(\cdot)^{1/K})$ of the ellipsoids generated by the three methods for different dimensions $K$.}
		\label{fig:chippedsqr}
 		\end{center}
 	\end{figure}
\end{example}
	
Next, we present the comparison of $\vol(\Ecop)$ with the volume of the ellipsoid provided by the \sproc described in Appendix \ref{appendix:sproc_mvep}. However, the application of \sproc requires an ellipsoidal constraint in addition to the affine inequalities that define the polytope $\Pc$ (see Remark \ref{rem:redundant_ellipsoid} in Appendix \ref{appendix:sproc}). This can be achieved by using any ellipsoid \mbox{$\E(\bm Q,\bm q) = \{\bm x \in \RR^K: \norm{\bm Q \bm x + \bm q}^2 \leq 1\}$} that contains the polytope $\Pc$, and adding $\norm{\bm Q \bm x + \bm q}^2 \leq 1$ as a redundant constraint in the definition of $\Pc$. The ellipsoid  $\E(\bm Q,\bm q)$ already serves as an approximation of $\Emve$. We can then apply the \sproc in the hope of finding an ellipsoid with lower volume; we use $\Esproc$ to denote this ellipsoid. However, in Proposition~\ref{prop:sproc_improvement}, presented below, we show that if the center of $\E(\bm Q,\bm q)$ lies inside $\Pc$, then applying the \sproc provides no improvement and, in fact, returns the ellipsoid $\Esproc = \E(\bm Q,\bm q)$ as its unique optimal solution. This result is counter-intuitive since the \sproc has been successfully applied in cases where $\Pc$ is defined as either the intersection or Minkowski sum of ellipsoids. Furthermore, if $\E(\bm Q,\bm q) = \Esmvie$ is used in the redundant quadratic constraint, then Proposition~\ref{prop:sproc_improvement} implies that the \sproc does not improve upon $\Esmvie$, since the center of $\Esmvie$ lies inside $\Pc$. In that case, $\vol(\Ecop) \leq \vol(\Esmvie) = \vol(\Esproc)$.

\begin{proposition}
	\label{prop:sproc_improvement}
	Let $\Pc$ be a polytope defined as in \eqref{eq:define_polytope} that satisfies Assumption~\ref{assum:set}, and let $\E(\bm Q, \bm q) = \{\bm x \in \RR^K: \norm{\bm Q \bm x + \bm q}^2 \leq 1\}$ be an ellipsoid containing $\Pc$ such that the center of $\E(\bm Q, \bm q)$ lies inside $\Pc$. Then, for the set $\{\bm x \in \RR^K:\ \bm S\bm x \leq \bm t, \norm{\bm Q \bm x + \bm q}^2 \leq 1 \}$, we have that $\Esproc = \E(\bm Q,\bm q)$.
\end{proposition}
{
\begin{proof}
	 See the e-companion.
\end{proof}

Finally, another method of approximating \eqref{eq:MVE} uses the decision rule approach described in \cite[Section 6]{zhen2017robust}, which we summarize in Appendix \ref{appendix:zrh}. We denote the ellipsoid generated using this approach by $\Ezrh$. In the following proposition, we show that $\vol(\Ezrh)$ is never lower than $\vol(\Esmvie)$. Thus, by Theorem~\ref{thm:compare_polytope}, $\vol(\Ezrh) \geq \vol(\Ecop)$.

\begin{proposition}
	\label{prop:compare_zrh_smvie}
	$\vol(\Ezrh) \geq \vol(\Esmvie)$.
\end{proposition}
\begin{proof} 
	See the e-companion.
\end{proof}

}

\subsection{Sets with Quadratic Constraints}
\label{sec:approx_quad}

Next, we provide a semidefinite programming approximation to \eqref{eq:MVE} when the set $\Pc$ is defined by affine, as well as quadratic inequalities. This generalizes the approximation \eqref{eq:MVE_approx_polytope} developed for the case of a polytope. Specifically, we consider the following set:
\begin{equation}\label{eq:quad_sets}\Pc = \left\{\bm x\in \RR^K\ : \ \bm S\bm x \leq \bm t,\ \norm{\bm Q_i \bm x + \bm q_i}^2 \leq 1 \ \forall i \in [I]\right\},\end{equation}
where $\bm S\in \RR^{J\times K},\;\bm t \in \RR^J,\; \bm Q_i \in \Sbb^K$ and $\bm q_i \in \RR^K$. In the next theorem, we derive sufficient conditions that an ellipsoid $\E(\bm A, \bm b)$ contains the set $\Pc$ defined as in \eqref{eq:quad_sets}.

\begin{theorem}
	\label{thm:C0_quad}
	Let the set $\Pc$ be defined as in \eqref{eq:quad_sets}. Consider any $\bm A \in \Sbb_{++}^K$ and $\bm b\in \RR^K$. Then, an ellipsoid $\E(\bm A, \bm b)$ contains $\Pc$ if there exist $\bm N \in \RR_+^{J\times J},\; \bm F\in \Sbb^K,\; \bm g\in \RR^K,\; h\in \RR,\; \lambda_i \geq 0\ \forall i \in [I],\;\bm\alpha_{ij}\in \RR^K, \kappa_{ij} \in \RR \  \forall i \in [I]\ \forall j \in [J]$ such that
	\begin{equation}
	\label{eq:suff_conds_quad}
	\begin{array}{clll}
	& \norm{\bm\alpha_{ij}} \leq \kappa_{ij} \;\; \forall i \in [I]\; \forall j \in [J], \\
	& \begin{bmatrix} \bm F & \bm g \\ \bm g^\top & h - 1
	\end{bmatrix} \displaystyle \preceq - \Sb^\top\bm N\Sb + \sum_{i \in [I]} \lambda_i \bm J_i - \sum_{i \in [I],j\in[J]} \bm M_{ij}(\bm\alpha_{ij},\kappa_{ij}), \gap \\
	&\begin{bmatrix}\bm F &\bm g & \bm A\\ \bm g^\top & h& \bm b^\top \\ \bm A& \bm b & \Ib \end{bmatrix} \succeq \bm 0,
	\end{array}
	\end{equation}
	where
	\begin{equation*}
	\begin{array}{clll}
	\displaystyle\Sb = \begin{bmatrix} -\bm S &\bm t \end{bmatrix} \in \RR^{J\times(K+1)},\;\; \bm J_i = \begin{bmatrix}  \bm Q_i^2 & \bm Q_i^\top \bm q_i \\ \bm q_i^\top \bm Q_i & \bm q_i^\top \bm q_i - 1 \end{bmatrix} \in \Sbb^{K+1} \;\; \forall i \in [I], \text{ and } \gap\\
	\displaystyle\bm M_{ij}(\bm\alpha,\kappa) = \begin{bmatrix}	-\frac{1}{2} \left(\bm S_{j:} \bm\alpha^\top\bm Q_i + \bm Q_i \bm \alpha\bm S_{j:}^\top\right) & \frac{1}{2} \left(t_j\bm Q_i \bm \alpha - (\bm \alpha^\top\bm q_i +\kappa)\bm S_{j:}\right) \\ \frac{1}{2} \left(t_j\bm Q_i \bm \alpha - (\bm \alpha^\top\bm q_i + \kappa)\bm S_{j:}\right)^\top & (\bm \alpha^\top\bm q_i + \kappa)t_j
	\end{bmatrix} \quad \forall i \in [I]\; \forall j \in [J].
	\end{array}
	\end{equation*}
\end{theorem}

\begin{proof}
	For the set $\Pc$, the cone $\K$ as defined as in \eqref{eq:def_cone} can be written as
	$$\K = \left\{ [\bm x; \tau] \in \RR^{K+1}: \; \tau \geq 0,\;\bm S\bm x \leq \tau \bm t,\; \norm{\bm Q_i \bm x + \tau \bm q_i}^2 \leq \tau^2 \ \forall i \in [I] \right\}.$$
	We show that the conditions \eqref{eq:suff_conds_quad} imply the conditions \eqref{eq:gc_conditions}, which proves the claim. Let 
	$$\bm P =\begin{bmatrix} \bm F&\bm g\\\bm g^\top & h - 1 \end{bmatrix}. $$
	Also, consider $[\bm x; \tau] \in \K$. From the first semidefinite inequality, we have that 
	$$\begin{bmatrix} \bm x\\ \tau \end{bmatrix}^\top \bm P \begin{bmatrix} \bm x\\ \tau \end{bmatrix} \leq \begin{bmatrix} \bm x\\ \tau \end{bmatrix}^\top \left( - \Sb^\top\bm N\Sb + \sum_{i \in [I]} \lambda_i \bm J_i - \sum_{i \in [I],j\in[J]} \bm M_{ij}(\bm\alpha_{ij})\right)\begin{bmatrix} \bm x\\ \tau \end{bmatrix}.$$	
	We show that all three terms in the expression on the right hand side are non-positive. The first term is non-positive as shown in the proof of Theorem \ref{thm:C0_polytope}. Next, observe that for all $i \in [I]$, we have that $[\bm x;\tau]^\top \bm J_i [\bm x; \tau] = \norm{\bm Q_i \bm x +\tau \bm q_i}^2 - \tau^2 \leq 0,$ since $[\bm x,\; \tau] \in \K$. Also, $$[\bm x;\tau]^\top \bm M_{ij}(\bm\alpha_{ij}) [\bm x; \tau] = (\tau t_j-\bm S_{j:}^\top\bm x)(\tau \kappa_{ij} + \bm\alpha_{ij}^\top(\bm Q_i\bm x + \tau\bm q_i)) \geq 0.$$ The previous inequality follows because both terms in the product are non-negative since $\bm S\bm x \leq \tau\bm t$ and $\tau \kappa_{ij} + \bm\alpha_{ij}^\top(\bm Q_i\bm x + \tau\bm q_i) \geq \tau \kappa_{ij} - \norm{\bm \alpha_{ij}}\norm{\bm Q_i\bm x + \tau\bm q_i} \geq \tau \kappa_{ij} - \tau \kappa_{ij} = 0.$ Hence, $[\bm x; \tau]^\top \bm P[\bm x; \tau] \leq 0\;\; \forall[\bm x;\tau] \in \K$, which implies that $\bm P \preceq_{\C(\K)}\bm 0.$	Hence the claim follows.
\end{proof}

Theorem \ref{thm:C0_quad} implies that the following SDP serves as a restriction to \eqref{eq:MVE}:
\begin{equation}
\label{eq:MVE_approx_quad}
\begin{array}{clll}
\minimize &-\log\det(\bm A)\\
\subjectto&\bm A\in \Sbb^K,\; \bm b \in \RR^K,\; \bm F\in \Sbb^K,\; \bm g\in \RR^K,\; h\in \RR,\\
& \bm N \in \RR_+^{J\times J},\; \lambda_i \geq 0\ \forall i \in [I],\;\bm\alpha_{ij}\in \RR^K,\\
& \eqref{eq:suff_conds_quad} \textnormal{ holds.}
\end{array}
\end{equation}

\begin{remark}
	The approximation discussed above is motivated by the Relaxation Linearization Technique (RLT) discussed in \cite{anstreicher2009semidefinite, sherali2013reformulation}, and SOC-RLT constraints discussed in \cite{burer2013second}.
\end{remark}

\section{Application to Distributionally Robust Optimization}
\label{sec:dro}
In this section, we demonstrate how our approximation to \eqref{eq:MVE} can be used to obtain good solutions to the two-stage DRO model with random recourse given by
\begin{equation}
\label{eq:dro_1ststage}
\begin{array}{clll}
\displaystyle \inf_{\bm x \in \X} \left\{ \bm c^\top \bm x + \sup_{\QQ\in\Qc} \EE_\QQ [\R(\bm x, \tilde{\bm \xi})] \right\},
\end{array}
\end{equation}
where 
\begin{equation}
\label{eq_dro_2ndstage}
\begin{array}{rlll}
\displaystyle\R(\bm x, \bm \xi) = \inf_{\bm y} & (\bm D\bm\xi+ \bm d)^\top \bm y\\
\st & \bm T_\ell(\bm x)^\top\bm \xi + h_\ell(\bm x) \leq \bm (\bm W_\ell\bm\xi+\bm w_\ell)^\top\bm y \quad \forall \ell \in [L].
\end{array}
\end{equation}
Here, $\bm x\in  \RR^{N_1}$ and $\bm y\in \RR^{N_2}$ represent the first- and the second-stage decision variables respectively, $\X$ is a set defined by tractable convex constraints on $\bm x$, and $\bm\xi\in \RR^K$ is the vector of uncertain parameters. Also, $\bm c\in \RR^{N_1}, \bm D\in \RR^{N_2\times K},\bm W_\ell \in \RR^{N_2\times K}, \bm d \in \RR^{N_2}$, and $\bm w_\ell \in \RR^{N_2}$ are problem parameters. The functions $\bm T_\ell:\X\rightarrow \RR^{K}$ and $h_\ell:\X\rightarrow \RR$ are affine in the input parameter. We consider the following moment-based ambiguity set: $\mathcal Q = \{ \QQ \in \Qc_0(\Xi): \EE_\QQ[\tilde{\bm \xi}] = \bm \mu,\, \EE_\QQ[\tilde{\bm \xi}\tilde{\bm \xi}^\top] \preceq \bm \Sigma \}$, where $\Xi = \{\bm \xi\in \RR^K: \bm S\bm \xi \leq \bm t\} $ is the bounded support set, and $\Qc_0(\Xi)$ is the set of all probability measures supported on $\Xi$. The objective function minimizes the sum of the first-stage and the expected recourse cost, where the expectation is taken with respect to the worst case distribution among those in the ambiguity set $\Qc$. The results presented here can be extended to other types of ambiguity sets, including the simpler case where $\Qc = \Qc_0(\Xi)$ (\ie, robust optimization) \cite{bertsimas2016multistage, xu2018copositive}, the more sophisticated data-driven Wasserstein ambiguity set \cite{hanasusanto2018conic}, and to the classical stochastic programming setting.

The problem \eqref{eq:dro_1ststage} can be written equivalently as:
\begin{equation}
\label{eq:dro_semiinf}
\begin{array}{clll}
\displaystyle \inf_{\bm x, \bm y(\cdot)} & \displaystyle \bm c^\top \bm x + \sup_{\QQ\in\Qc} \EE_\QQ [(\bm D\bm\xi+ \bm d)^\top \bm y(\bm\xi)] ,\\
\st & \bm x \in \X,\\
 & \bm T_\ell(\bm x)^\top\bm \xi + h_\ell(\bm x) \leq \bm (\bm W_\ell\bm\xi+\bm w_\ell)^\top\bm y(\bm \xi) \quad \forall \bm\xi \in \Xi, \forall \ell \in [L],
\end{array}
\end{equation}
where the second-stage decision variable $\bm y$ is a function of the uncertain parameters $\bm\xi$. The problem \eqref{eq:dro_semiinf} is difficult to solve. To generate a tractable approximation to \eqref{eq:dro_semiinf}, we explore the use of piecewise-linear decision (PLD) rules. Specifically, we partition $\Xi$ into regions $\Xi_1,\ldots, \Xi_J$, and restrict $\bm y(\cdot)$ to be of the form $\bm y(\bm\xi) = \bm Y_j \bm \xi + \bm y_j$ if ${\bm\xi \in \Xi_j}$, where $\bm Y_j\in \RR^{N_2\times K}$ and $\bm y_j\in \RR^{N_2}$.  For constructing the partitions, we start with a set of \emph{constructor points} $\{\bm \xi_j\}_{j\in [J]}$ in $\Xi$. Then, we define the partition $\Xi_j$ to be the set of all points in $\Xi$ which are closer to $\bm\xi_j$ than any other constructor point. In other words,
\begin{equation*}
\begin{array}{cll}
\Xi_j & = \{ \bm \xi \in \RR^K : \bm S\bm \xi \leq \bm t,\; \norm{\bm \xi - \bm \xi_j} \leq \norm{\bm \xi- \bm \xi_i}\; \forall i \in [J], i \neq j \}\\
& = \{ \bm \xi \in \RR^K : \bm S\bm \xi \leq \bm t,\; 2(\bm\xi_i-\bm\xi_j)^\top \bm\xi \leq \bm\xi_i^\top\bm\xi_i - \bm\xi_j^\top\bm\xi_j \; \forall i \in [J], i \neq j \}\\
& = \{ \bm \xi \in \RR^K : \bm S_j \bm \xi \leq \bm t_j \},
\end{array}
\end{equation*}
where the matrix $\bm S_j \in \RR^{L_j\times K}$ and the vector $\bm t_j\in \RR^{L_j}$ are formed by combining the linear constraints in the definition of $\Xi_j$. These partitions are known as \emph{Voronoi regions}. 

Because of random recourse (\ie, uncertainty in the coefficients of $\bm y(\cdot)$), finding the optimal PLD rule is NP-hard, even if there is only one piece \cite{BGGN04:LDR}. However, we can approximate the problem of finding the optimal PLD rule using the \sproc. {However we need a quadratic constraint in the definition of $\Xi_j$ for an effective application of \sproc (see Remark \ref{rem:redundant_ellipsoid} in Appendix \ref{appendix:sproc}). To this end, let $\E(\bm A_j,\bm b_j)$ be an ellipsoid that contains $\Xi_j$. Since $\Xi_j$ is a polytope, we can exploit the results developed in Section \ref{sec:approx_polytope} to find $\E(\bm A_j,\bm b_j)$. We can write $\Xi_j$ equivalently as $\Xi_j = \{ \bm \xi \in \RR^K : \bm S_j \bm \xi \leq \bm t_j, \norm{\bm A_j \bm \xi + \bm b_j}^2 \leq 1 \}$.} We illustrate the procedure of partitioning and covering with ellipsoids in Figure \ref{fig:voronoi}.

\begin{figure}
	\vspace{-1cm}
	\begin{center}
		\begin{subfigure}{.1\textwidth}
		\end{subfigure}\hfill
		\begin{subfigure}{.4\textwidth}
			\includegraphics[width=2in]{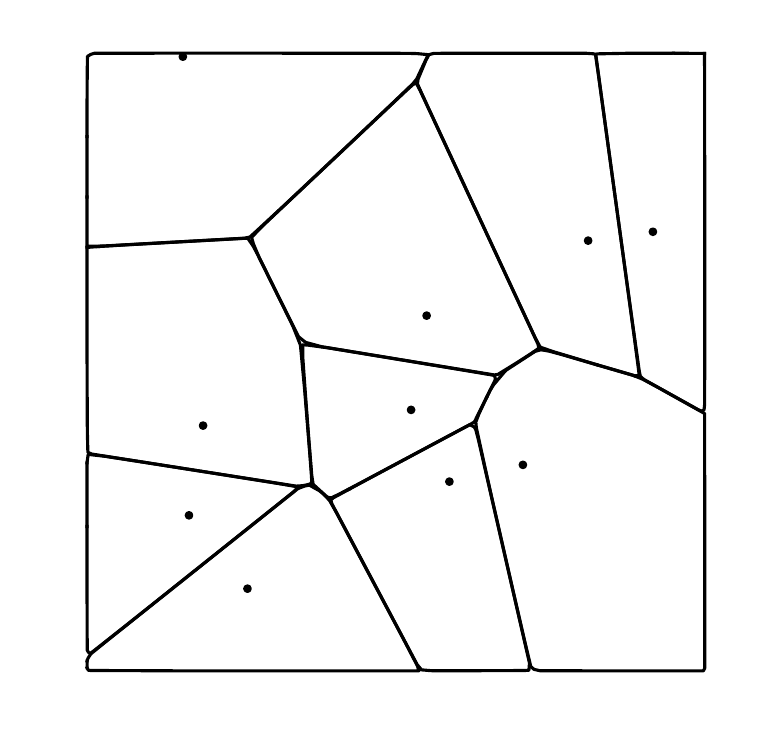}
		\end{subfigure}\hfill
		\begin{subfigure}{.4\textwidth}
			\includegraphics[width=2.4in]{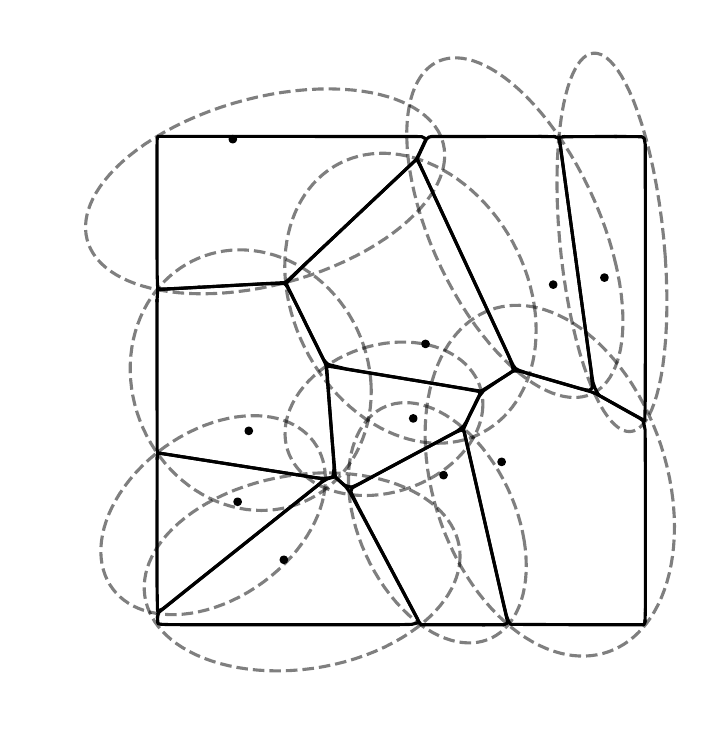}
		\end{subfigure}\hfill
		\begin{subfigure}{.1\textwidth}
		\end{subfigure}\hfill
	\end{center}
	\vspace{-5mm}
	\caption{Voronoi Regions: The outer square represents the support set, and the black dots are the constructor points. The points are used to construct partitions, and an ellipsoid containing each partition is found by solving the SDP \eqref{eq:MVE_approx_polytope}. \label{fig:voronoi}}
\end{figure}

In the next proposition, we derive a tractable SDP that generates a feasible PLD rule. The optimal value of the resulting SDP approximation provides an upper bound to the optimal value of~\eqref{eq:dro_semiinf}. In Example \ref{example:dro2} presented after the proposition, we demonstrate that the size of the bounding ellipsoids $\E(\bm A_j,\bm b_j)$ can drastically impact the upper bound provided by the SDP approximation; in particular, the tighter the ellipsoids, the better the upper bound.

\begin{proposition}
\label{prop:dro}
Consider the following SDP:
\begin{equation}
\label{eq:dro_pldr}
\begin{array}{rlll}
\inf & \bm c^\top \bm x + \alpha + \bm \beta^\top \bm \mu + \tr(\bm \Gamma \bm\Sigma)\\
\st & \bm x \in \X,\; \bm\Gamma\in \Sbb_+^K,\; \bm \beta\in \RR^K,\;\bm\alpha\in \RR,\;\\
& \bm Y_j \in \RR^{N_2\times K},\; \bm y_j \in \RR^{N_2},\; \bm\gamma_j\in \RR_+^{L_j},\; \delta_j \in \RR_+ \quad \forall j \in [J],\\
& \lambda_{j\ell} \in \RR_+,\; \bm \rho_{j\ell} \in \RR_+^{L_j} \quad \forall j \in [J]\; \forall \ell \in [L],\vspace{1mm} \\
& \displaystyle \begin{bmatrix} \bm\Gamma & \frac{1}{2}\bm\beta\\ \frac{1}{2}\bm \beta^\top & \alpha \end{bmatrix} - \begin{bmatrix} \frac{1}{2}(\bm D^\top \bm Y_j + \bm Y_j^\top \bm D) & \frac{1}{2}(\bm D^\top \bm y_j + \bm Y_j^\top \bm d ) \\ \frac{1}{2}( \bm D^\top \bm y_j + \bm Y_j^\top \bm d  )^\top &  \bm d^\top \bm y_j  \end{bmatrix} + \bm P_j(\bm \gamma_j) + \delta_{j} \bm J_j \succeq \bm 0\quad \forall j \in [J], \vspace{1mm}\\
&\displaystyle \begin{bmatrix} \frac{1}{2}(\bm W_\ell^\top \bm Y_j + \bm Y_j^\top \bm W_\ell) & \frac{1}{2}(\bm W_\ell^\top \bm y_j + \bm Y_j^\top \bm w_\ell ) \\ \frac{1}{2}( \bm W_\ell^\top \bm y_j + \bm Y_j^\top \bm w_\ell  )^\top &  \bm w_\ell^\top \bm y_j  \end{bmatrix} - \bm M_\ell(\bm x) + \bm P_j(\bm \rho_{j\ell}) + \lambda_{j\ell} \bm J_j \succeq \bm 0\quad \forall j \in [J]\; \forall \ell \in [L],
\end{array}
\end{equation}
where
$$\bm M_\ell(\bm x) = \begin{bmatrix} \bm 0 & \frac{1}{2} \bm T_\ell(\bm x) \\ \frac{1}{2} \bm T_\ell(\bm x)^\top & h_\ell(\bm x) \end{bmatrix},\quad \bm P_j(\bm\rho) = \begin{bmatrix} \bm 0 & \frac{1}{2}\bm S_j^\top \bm \rho \\ \frac{1}{2}\bm \rho^\top \bm S_j & - \bm t_j^\top \bm \rho \end{bmatrix},\quad \mbox{and}\quad \bm J_j = \begin{bmatrix}  \bm A_j^2 & \bm A_j^\top \bm b_j \\ \bm b_j^\top \bm A_j & \bm b_j^\top \bm b_j - 1 \end{bmatrix}.$$
Let $\bm y(\bm \xi) = \bm Y_j\bm \xi + \bm y_j $ if $\bm \xi \in \Xi_j$. Then, $(\bm x, \bm y(\cdot))$ provides a feasible solution to \eqref{eq:dro_semiinf}. Furthermore, the optimal value of \eqref{eq:dro_pldr} provides an upper bound to the optimal value of \eqref{eq:dro_semiinf}.
\end{proposition}
\begin{proof}
	See e-companion.
\end{proof}


\begin{example}
	\label{example:dro2}
	Consider the following special case of \eqref{eq:dro_semiinf}:
	\begin{equation}
	\begin{array}{rlll}
	\label{eq:dro_ex2}
	\displaystyle z = \inf_{x,\bm y(\cdot)} & x\\
	\st & 1 \leq (\bm \xi+\ee)^\top \bm y(\bm \xi) \leq x \quad \forall \bm \xi \in \Xi,
	\end{array}
	\end{equation}
	where $\Xi = \{ \bm\xi \in \RR^K: \bm 0 \leq \bm\xi \leq \ee \}$ is the unit hypercube, and $J=1$. {This problem is a special case of \eqref{eq:dro_semiinf} with $L = 2, \ \bm D = \bm 0,\ \bm d = \bm 0,\ \bm W_1 = \Ib,\ \bm w_1 = \ee,\ \bm T_1(x) = \bm 0,\ h_1(x) = 1,\ \bm W_2 = -\Ib,\ \bm w_2 = -\ee,\ \bm T_2(x) = \bm 0,\textnormal{ and } h_2(x) = -x.$} The true optimal value is $z = 1$, which is obtained by the non-linear decision function $\bm y(\bm \xi) = (\bm \xi+\ee)/\norm{\bm \xi+\ee}^2$. In this case, $\Emve = \{ \bm \xi\in \RR^K : \norm{\bm\xi -\ee/2}^2 \leq N/4 \}$. For $s \geq 0$, let $z(s)$ be the upper bound generated by the SDP approximation when $\{ \bm\xi \in \RR^K : \norm{\bm\xi -\ee/2}^2 \leq N(1 + s)/4 \}$ is used as the bounding ellipsoid. In the e-companion, we show that
	$$z(s) = \begin{cases}
	9/(8-s) &\textnormal{if } 0\leq s \leq 2,\\
	1+s/4 &\textnormal{if } 2\leq s \leq 4,\\
	2 & \textnormal{if } 4\leq s.
	\end{cases}$$
	Therefore, the linear decision rule obtained with $\Emve$ generates an objective value of $z(0) = 9/8 = 1.125$. {The objective value $z(s)$ increases as the size of the ellipsoid increases. The case when $s$ approaches $\infty$ corresponds to dropping the ellipsoidal constraint; in that case, we obtain an objective value of $\lim_{s \to \infty} z(s) = 2$. Hence, ignoring the ellipsoidal constraint can increase the suboptimality of the decision rules approximation from $12.5\%$ to $100\%$.}
\end{example}

Example \ref{example:dro2} demonstrates the importance of generating good outer ellipsoids. We further elaborate on this point in Section \ref{sec:inventory}, where we perform experiments on randomly generated instances of an inventory management model. We note that the task of finding the outer ellipsoids $\E(\bm A_j,\bm b_j)$ can be parallelized, which leads to a substantial reduction in the computation time.
{
	\begin{remark}[Two-Stage Stochastic Programming]
		\label{ref:two_stage}
		In the classical stochastic programming setting, the random parameters  $\tilde{\bm \xi}$ are assumed to be governed by a known distribution $\PP$. The semi-infinite constraints in~$\eqref{eq:dro_semiinf}$ remain unchanged and can be approximated   in the  same manner using the \sproc. On the other hand, the worst-case expectation in the objective function of \eqref{eq:dro_semiinf} reduces to the expectation $\EE_\PP [(\bm D\tilde{\bm\xi}+ \bm d)^\top \bm y(\tilde{\bm\xi})]$. Applying the law of total expectation and employing the proposed PLD rules, we can reformulate the  expectation as 
		\begin{equation*}
		\begin{array}{rl}
		&\displaystyle\sum_{j\in[J]}	\PP(\tilde{\bm\xi}\in\Xi_j)\;\EE_\PP \left[(\bm D\tilde{\bm\xi}+ \bm d)^\top \bm y(\tilde{\bm\xi})~\big|~\tilde{\bm\xi}\in\Xi_j\right]	\\
		=&\displaystyle\sum_{j\in[J]}	\PP(\tilde{\bm\xi}\in\Xi_j)\;\EE_\PP \left[(\bm D\tilde{\bm\xi}+ \bm d)^\top (\bm Y_j \tilde{\bm\xi}+ \bm y_j) ~\big|~\tilde{\bm\xi}\in\Xi_j\right]\\
		=&\displaystyle\sum_{j\in[J]}	\PP(\tilde{\bm\xi}\in\Xi_j)\left(\tr\left(\bm D^\top\bm Y_j\EE_\PP[\tilde{\bm\xi}\tilde{\bm\xi}^\top|\tilde{\bm\xi}\in\Xi_j ]\right)+(\bm y_j^\top\bm D+\bm d^\top\bm Y_j)\EE_\PP[\tilde{\bm\xi}|\tilde{\bm\xi}\in\Xi_j ] + \bm d^\top\bm y_j\right).	
		\end{array}
		\end{equation*}
		This expression is affine in the decision variables $\bm Y_j$ and $\bm y_j$, $j\in[J]$. Note that the partition probabilities $\PP(\tilde{\bm\xi}\in\Xi_j)$, $j\in[J]$, and conditional moments $\EE_\PP[\tilde{\bm\xi}|\tilde{\bm\xi}\in\Xi_j ]$ and $\EE_\PP[\tilde{\bm\xi}\tilde{\bm\xi}^\top|\tilde{\bm\xi}\in\Xi_j ]$, $j\in[J]$, can be estimated using the Monte Carlo sampling method. 
	\end{remark}
}

\section{Numerical Experiments}
\label{sec:experiments}

In this section, we present numerical experiments that demonstrate the improved performance of our scheme for approximating \eqref{eq:MVE} over the existing methods. First, we show that our approach outperforms the existing approaches in terms of solution quality and computational time on randomly generated polytopes. Second, we demonstrate the efficacy of our method in generating quality solutions for a distributionally robust inventory management model.  All optimization problems are solved using the YALMIP interface \cite{yalmip} on a 16-core 3.4 GHz computer with 32 GB RAM. We use MOSEK 8.1 to solve SDPs and CPLEX 12.8 to solve non-convex quadratic programs to optimality. 

\subsection{Random Polytopes}
\label{sec:random_polytopes}
Here, we compare our method of approximating \eqref{eq:MVE} with \begin{inparaenum}[(i)] \item the constraint generation approach \cite{gotoh2006minimal}, \item the SMVIE approach, and \item the method using sufficient conditions proposed by Kellner, Theobald, and Trabandt  \cite{kellner2013containment}. \end{inparaenum} We refer to the last method as the KTT approach, and denote the corresponding ellipsoid by $\Ektt$ (see Appendix \ref{appendix:ktt} for details on the formulation).

For our experiments, we generate polytopes randomly as follows. We start with the hyper-rectangle $\{\bm x \in \RR^K :\ \bm 0\leq \bm x \leq \ee \}$ with center $\mathbf c = \ee/2$. Then we add $M$ linear inequalities in the following way. For $j \in [M]$, we generate a vector $\bm s_j\in \RR^K$ uniformly distributed on the surface of the unit hypersphere. We generate a distance $r_j$ uniformly at random from the interval $[-\norm{\bm s_j}_1/2,\norm{\bm s_j}_1/2]$, and add the constraint $\bm s_j^\top (\bm x - \mathbf c) \leq r_j$ if $r_j>0$ and $\bm s_j^\top (\bm x - \mathbf c) \geq r_j$ if $r_j\leq 0$. Choosing $r_j$ from the specified interval leads to a constraint that cuts the hyper-rectangle (\ie, the constraint is not redundant). Also, the construction ensures that the polytope is non-empty since $\mathbf c$ satisfies all the constraints.

For several values of $K$, we solve the problem exactly and apply each approximation method on $50$ randomly generated instances for $M = K,\ 2K,\ 3K$. We report the suboptimality results of the three approximation methods in Table~\ref{table:polytope_subopt}. For higher values of $K$, for which we were not able to solve the problem exactly within $30$ minutes, we report the suboptimality of the radius of $\Esmvie$ and $\Ektt$ with respect to $\Ecop$ in Table~\ref{table:subopt_big}. Finally, the solution times of different methods are reported in Table \ref{table:polytope_times}. We do not report the solution time of the SMVIE approach. Even for the largest problem size that we solved, the SMVIE approach produces solutions in less than $2$ seconds, dominating every other approach.

It can be observed that the radius (and therefore, volume) of $\Ecop$ is significantly lower than that of $\Esmvie$. Furthermore, the suboptimality of the radius of $\Esmvie$ relative to that of $\Ecop$ increases with the dimension $K$ (from $246\%$ for $K=15$ to $481\%$ for $K=40$). This is perhaps because the scale factor of $K$ becomes very conservative for higher values of $K$. This increase in solution quality of $\Ecop$ comes at the cost of higher solution times compared to that of finding $\Esmvie$.

We also observe that the radius of $\Ecop$ is slightly better than that of $\Ektt$; the solution time, however, is significantly lower ($1$-$2$ orders of magnitude). As an example, for $K=M=30$, the KTT approach does not provide solutions within $30$ minutes, whereas our method generates an solution in $13.7$ seconds on average.

Finally, we observe that for small problem instances, our method finds a solution much faster than solving the problem to optimality. For higher dimensional problems ($K>15$), where solving the problem exactly becomes intractable, our approximation continues to provide ellipsoids of lower volume than the other approximation methods.

\begin{table}
	\centering
		\begin{tabular}{c|rrr|rrr|rrr}
			\multicolumn{1}{c}{} &
			\multicolumn{3}{c}{$M=K$} & \multicolumn{3}{c}{$M=2K$} &  	\multicolumn{3}{c}{$M=3K$}\\ \hline
			$K$  & Copos  &  KTT   & SMVIE  & Copos  &  KTT   & SMVIE  & Copos  &  KTT   & SMVIE \\
			[0.0mm] \hline
			$2$  & 3.41\% & 4.68\% & 34.3\% & 5.20\% & 6.48\% & 32.8\% & 5.33\% & 6.63\% & 31.9\% \\
			$5$  & 4.88\% & 7.02\% & 105\% & 9.92\% & 13.16\% & 91.9\% & 13.2\% & 16.4\% & 93.7\% \\
			$10$ & 2.53\% & 3.72\% & 188\% & 7.48\% & 9.51\% & 176\% & 13.6\% & 16.9\% & 164\% \\ 
			$15$ & 1.29\% & 1.84\% & 250\% & 5.57\% & 7.16\% & 230\% & N/A & N/A & N/A\\
			\hline\hline
	\end{tabular}
	\caption[Suboptimality for minimum volume ellipsoids]{Random Polytopes: Mean suboptimality of the radii of $\Ecop$ (`Copos'), $\Ektt$ ('KTT'), and $\Esmvie$ (`SMVIE') for different problem sizes. We use `N/A' when the problem cannot be solved to optimality within $30$ minutes.
		\label{table:polytope_subopt}}
\end{table}

\begin{table}
	\centering
	\begin{tabular}{c|rr|rr|rr}
		\multicolumn{1}{c}{} &
		\multicolumn{2}{c}{$M=K$} & \multicolumn{2}{c}{$M=2K$} &  	\multicolumn{2}{c}{$M=3K$}\\ \hline
		$K$  &  KTT   & SMVIE  &  KTT   & SMVIE  &  KTT   & SMVIE \\
		[0.0mm] \hline
		$15$  & 0.54\% & 246\% & 1.50\% & 212\% &  2.07\% & 191\% \\
		$20$  & 0.30\% & 310\% & 1.01\% & 268\% & 1.65\% & 245\%  \\
		$25$  & 0.28\% & 357\% & 0.66\% & 318\% &   --   & 292\% \\ 
		$30$   &   --  & 401\% &   --   & 364\% &   --   & 329\% \\ 
		$35$   &   --  & 440\% &   --   & 405\% &   --   & 372\% \\ 
		$40$   &   --  & 481\% &   --   & 447\% &   --   & 414\% \\ 
		\hline\hline
	\end{tabular}
	\caption[Relative suboptimality for minimum volume ellipsoids]{Random Polytopes: Mean suboptimality of the radii of $\E_{\textnormal{ktt}}$ (`KTT') and $\Esmvie$ (`SMVIE') relative to $\Ecop$ for the cases which could not be solved to optimality within $30$ minutes. We use ``--'' for the cases when the KTT approach does not provide a solution within $30$ minutes.
		\label{table:subopt_big}}
\end{table}

\begin{table}
	\centering
	\begin{tabular}{c|rrr|rrr|rrr}
		\multicolumn{1}{c}{} &
		\multicolumn{3}{c}{$M=K$} & \multicolumn{3}{c}{$M=2K$} &  	\multicolumn{3}{c}{$M=3K$}\\ \hline
		$K$  & Exact & Copos  & KTT  & Exac & Copos &  KTT  & Exact & Copos &  KTT  \\[0.0mm] \hline
		$2$  & 1.52 & 0.004  & 0.011 & 1.53 & 0.005  & 0.027 & 1.69 &  0.005 & 0.059 \\
		$5$  & 8.56 & 0.014  & 0.036 & 9.13 & 0.023  & 0.073 & 9.59 &  0.050 & 0.096 \\
		$10$ & 72.6 & 0.106  & 0.925 & 81.7 & 0.290  & 2.09 & 133 &  0.754 & 3.78 \\
		$15$ & 406  & 0.542  & 10.0  & 1191 & 1.82   & 25.8 & -- & 5.21 & 49.7 \\
		$20$ & -- & 2.01 & 73.2 &-- &7.60 &210 &-- & 22.2 & 438 \\
		$25$ & -- & 5.65 & 368 &-- & 22.8 & 1067 &--& 68.0 & -- \\
		$30$ & -- & 13.7 & -- &-- & 54.7 & -- &--& 207 & -- \\
		$35$ & -- & 28.8 & -- &-- & 133 & -- &--& 492 & -- \\
		$40$ & -- & 53.2 & -- &-- & 302 & -- &--& 1155 & -- \\
		\hline\hline
	\end{tabular}
	\caption[Solution times for minimum volume ellipsoids]{Random Polytopes: Mean solution times (in seconds) of the exact method (`Exact'), our proposed method (`Copos'), and the KTT approach (`KTT') for different problem sizes. We use ``--'' when the corresponding method does not provide a solution within $30$ minutes. \label{table:polytope_times}}
\end{table}

\subsection{Risk-Averse Inventory Management}
\label{sec:inventory}

Next, we consider an inventory management problem, where we decide the purchase amount of $N$ products before observing their demands. We incur a holding cost if we purchase more than  the demand, and a stockout cost if we purchase less than the demand. We assume that the demands and the stockout costs are random. The objective is to minimize the worst-case conditional value at risk (CVaR) \cite{RU00:CVaR, zhu2009worst, natarajan2009constructing} of the total cost. We can write the model as follows:

\begin{equation*}
\begin{array}{cll}
\displaystyle \minimize &\displaystyle \sup_{\QQ \in \mathcal{Q}} \QQ\mbox{-}\textnormal{CVaR}_\epsilon [\mathcal R(\bm x,\tilde{\bm \xi}, \tilde{\bm s})] \\
\subjectto & \bm x \in \RR^N,\; \bm x \geq \bm 0,\; \ee^\top \bm x \leq B,
\end{array}
\end{equation*}
where
\begin{equation*}
\begin{array}{rll}
\displaystyle \mathcal R(\bm x, \bm \xi, \bm s) = \inf & \bm g^\top \bm y_1 + \bm s^\top \bm y_2\\
\st & \bm y_1\in \RR_+^N,\bm y_2\in\RR_+^N,\\
& \bm y_1 \geq \bm x - \bm\xi,\;\; \displaystyle \bm y_2 \geq \bm\xi - \bm x.
\end{array}
\end{equation*}
Here, the variables $\bm x$, $\bm y_1$ and $\bm y_2$ represent the vector of purchase decisions, excess amounts and shortfall amounts, respectively. The vector $\bm g\in \RR^N$ represents the known holding costs, and $B$ denotes budget on the total purchase amount. Also, $\bm \xi\in \RR^N$ and $\bm s\in \RR^N$ are random parameters which represent the vectors of demand and stock-out costs respectively. The ambiguity set $\Qc$ is as described in Section \ref{sec:dro}. By employing the definition of CVaR, it can be shown that the above problem is equivalent to
\begin{equation}
\label{eq:inventory}
\begin{array}{cll}
\displaystyle \minimize & \displaystyle \kappa + \frac{1}{\epsilon} \sup_{\QQ \in \Qc} \EE_\QQ [\tau(\tilde{\bm \xi},\tilde{\bm s})] \gap\\
\subjectto & \ \kappa \in \RR,\; \bm x \in \RR^N,\; \bm x \geq \bm 0,\; \mathbf e^\top \bm x \leq B,\\
& \displaystyle \left. \begin{array}{l}
\mspace{-6mu} \tau(\bm\xi,\bm s) \geq 0,\;\bm y_1(\bm\xi,\bm s) \geq \bm 0,\; \bm y_2(\bm\xi,\bm s) \geq \bm 0,\\
\mspace{-6mu} \tau(\bm\xi,\bm s) \geq \bm g^\top \bm y_1(\bm\xi,\bm s) + {\bm s}^\top \bm y_2(\bm\xi,\bm s) - \kappa, \\
\mspace{-6mu} \bm y_1(\bm\xi,\bm s) \geq \bm x - \bm\xi,\;\; \bm y_2(\bm\xi,\bm s) \geq \bm\xi - \bm x
\end{array} \mspace{50mu} \right \rbrace &\forall (\bm\xi,\bm s) \in \Xi,
\end{array}
\end{equation}
which is of the form \eqref{eq:dro_semiinf} \cite{hanasusanto2015distributionally, shapiro2002minimax}.

We generate the parameters for this problem as follows. We use $N = 7$ products, which leads to $2N = 14$ random parameters. We choose $\Xi = \{ [\bm\xi;\bm s] : \bm \xi_l \leq \bm\xi\leq \bm \xi_u, \bm s_l \leq \bm s \leq \bm s_u \}$, and $\epsilon = 5\%$. We partition $\Xi$ into $J=4$ regions, and select the constructor points $\{ [\bm\xi_j;\bm s_j] \}_{j\in [J]}$ by sampling uniformly at random from $\Xi$. We choose $B = 30$, $\bm \xi_l = \bm 0$, $\bm\xi_u = 10\ee$, $\bm s_l = 8\ee$, $\bm s_u = 12\ee$. For constructing the ambiguity set, we use $\bm\mu = [\bm \mu_{\bm\xi};\bm \mu_{\bm s}] \in \RR^{2N}$, where $\bm \mu_{\bm s} = 10\ee$ and every element of $\bm \mu_{\bm\xi}$ is generated uniformly from the interval $[0,2]$. We select a random correlation matrix $\bm C\in \Sbb_+^{2N}$ with the MATLAB command ``$\verb|gallery(`randcorr',2*N)|$'', and set $\bm\Sigma = \Diag(\bm\sigma)\bm C\Diag(\bm \sigma) + \bm \mu \bm \mu^\top$, where $\bm\sigma = [\bm \sigma_{\bm\xi};\bm \sigma_{\bm s}] \in \RR^{2N}$, $\bm \sigma_{\bm s} = \ee/2$ and $\bm \sigma_{\bm\xi} = \bm \mu_{\bm\xi}/4$.

We approximate \eqref{eq:inventory} using our proposed SDP \eqref{eq:dro_pldr}, where the ellipsoids $\E(\bm A_j,\bm b_j)$ are generated using the SDP \eqref{eq:MVE_approx_polytope} developed in Section \ref{sec:approx_polytope}. We refer to this approach here as `PWL'. We compare the solution time and quality of the PWL approach with those of the following schemes:
\begin{itemize}
	\item Piecewise static decision rules (`PWS') \cite{bertsimas2016multistage}: {Here, the second stage decision variables are restricted to be constant within each partition, \ie, $\bm Y_j = \bm 0$ in Proposition \ref{prop:dro}. This approach leads to a tractable approximation, and, to the best of our knowledge, is state-of-the-art for solving DRO problems with random recourse.}
	\item Linear decision rules (`LDR'): This is similar to PWL except we do not partition the support set (\ie, $J = 1$). We compare against LDR to demonstrate the advantage of partitioning the support set.
	\item Ellipsoids of double radius (`PWL-2'): To demonstrate the importance of the size of the ellipsoid, we present comparisons against the scheme similar to PWL, except we double the radii of the ellipsoids $\E(\bm A_j,\bm b_j)$ used in PWL.
\end{itemize}
We perform the experiment on $100$ randomly generated instances, and present the relative objective gaps in Table \ref{table:inventory_obj}. We also report the average solution times in Table \ref{table:inventory_time}. We assume that we can parallelize the task of generating the ellipsoids for each partition on $4$ machines. Since we consider $J=4$, for the solution time of the PWL approach, we choose the maximum among the solution times to find the $4$ ellipsoids, and add that to the solution time of solving the SDP \eqref{eq:dro_pldr}.

{The results indicate that we outperform the other methods in terms of the quality of the approximation. We observe that neglecting the linear term in the decision rules (\ie, using static decision rules) can lead to $75\%$ increase in the objective value. Thus, although static decision rules lead to a tractable formulation that requires less computational time, they also generate significantly worse solutions. Furthermore, not partitioning the support set can lead to $24\%$ higher objective values. Finally, doubling the radii of the bounding ellipsoids can increase the objective by $47\%$. For two-stage DRO models with random recourse, these results exhibit the importance of \begin{inparaenum}[(i)] \item using piecewise \emph{linear} instead of piecewise static decision rules, \item partitioning the support set, and \item having good ellipsoidal approximations to the partitions of the support set.\end{inparaenum} The improvement in solution quality comes at the expense of increased computational time. However, if one is willing to spend computational resources, significant improvement in the solution quality can be achieved by using our method.}

\begin{table}
	\centering
	\begin{tabular}{c|rrr}
		\hline
		Statistic       &  PWS   &  LDR   & PWL-2  \\ \hline
		Mean            & 75.1\% & 24.5\% & 47.4\% \\
		10th Percentile & 33.3\% & 1.23\% & 25.6\% \\
		90th Percentile & 130\%  & 49.4\% & 71.4\% \\
		\hline\hline
	\end{tabular}
	\caption{Inventory Management: Objective gaps of other models relative to PWL model. \label{table:inventory_obj}}
\end{table}

\begin{table}
	\centering
	\begin{tabular}{c|rrrr}
		\hline
		Statistic          & PWL & PWS  & LDR & PWL-2 \\ \hline
		Solution Time (ms) & 622 & 91.8 & 219 & 617 \\
		\hline\hline
	\end{tabular}
	\caption{Inventory Management: Average solution times of the models (in milliseconds). \label{table:inventory_time}}
\end{table}

\section{Conclusions}
\label{sec:conc}
In this article, we propose a GC reformulation for the minimum volume ellipsoid problem. We use that reformulation to generate tractable approximations when the set is defined by affine and quadratic inequalities. We prove the volume of the ellipsoids that our approach provides never exceeds the volume of $\Esmvie$. Furthermore, we demonstrate empirically that our method performs better than the other competing schemes for providing approximate solutions to the MVEP, in terms of solution time and quality. Finally, we use our method to efficiently generate high-quality approximations in the context of distributional robust optimization and linear dynamical systems.

The work presented in this paper leaves room for further investigation. First, it would be interesting to study the suboptimality bounds of the radii of the ellipsoids generated by our method. In particular, for $\Esmvie$, it is known that $\mbox{Radius}(\Esmvie) \leq K\cdot\mbox{Radius}(\Emve)$. It would be interesting to see if a better upper bound can be proved for the radius of $\Ecop$. A second possible direction is to utilize the GC reformulation to generate approximation for other types of sets. Studying such approximations would add to the entire copositive programming literature, and not only to the minimum volume ellipsoid problem. 

\subsubsection*{Acknowledgments}
This research was supported by the National Science Foundation grant no.\ 1752125.
\newpage

\bibliographystyle{plain}
\bibliography{bibliography}

\begin{thebibliography}{10}

\bibitem{ahipacsaouglu2015fast}
S.~D. Ahipa{\c{s}}ao{\u{g}}lu.
\newblock Fast algorithms for the minimum volume estimator.
\newblock {\em Journal of Global Optimization}, 62(2):351--370, 2015.

\bibitem{anstreicher2009semidefinite}
K.~M. Anstreicher.
\newblock Semidefinite programming versus the reformulation-linearization
  technique for nonconvex quadratically constrained quadratic programming.
\newblock {\em Journal of Global Optimization}, 43(2-3):471--484, 2009.

\bibitem{BGGN04:LDR}
A.~Ben-Tal, A.~Goryashko, E.~Guslitzer, and A.~Nemirovski.
\newblock Adjustable robust solutions of uncertain linear programs.
\newblock {\em Mathematical Programming A}, 99(2):351--376, 2004.

\bibitem{ben2001lectures}
A.~Ben-Tal and A.~Nemirovski.
\newblock {\em Lectures on modern convex optimization: {A}nalysis, algorithms,
  and engineering applications}, volume~2.
\newblock {SIAM}, 2001.

\bibitem{ben2002robust}
A.~Ben-Tal, A.~Nemirovski, and C.~Roos.
\newblock Robust solutions of uncertain quadratic and conic-quadratic problems.
\newblock {\em SIAM Journal on Optimization}, 13(2):535--560, 2002.

\bibitem{bertsimas2010models}
D.~Bertsimas, X.~V. Doan, K.~Natarajan, and C.-P. Teo.
\newblock Models for minimax stochastic linear optimization problems with risk
  aversion.
\newblock {\em Mathematics of Operations Research}, 35(3):580--602, 2010.

\bibitem{bertsimas2016multistage}
D.~Bertsimas and I.~Dunning.
\newblock Multistage robust mixed-integer optimization with adaptive
  partitions.
\newblock {\em Operations Research}, 64(4):980--998, 2016.

\bibitem{bomze2012copositive}
I.~M. Bomze.
\newblock Copositive optimization--recent developments and applications.
\newblock {\em European Journal of Operational Research}, 216(3):509--520,
  2012.

\bibitem{boyd1994linear}
S.~Boyd, L.~El~Ghaoui, E.~Feron, and V.~Balakrishnan.
\newblock {\em Linear matrix inequalities in system and control theory},
  volume~15.
\newblock {SIAM}, 1994.

\bibitem{boyd2004convex}
S.~Boyd and L.~Vandenberghe.
\newblock {\em Convex optimization}.
\newblock Cambridge {U}niversity {P}ress, 2004.

\bibitem{burer2012copositive}
S.~Burer.
\newblock Copositive programming.
\newblock In {\em Handbook on Semidefinite, Conic and Polynomial Optimization},
  pages 201--218. Springer, 2012.

\bibitem{burer2013second}
S.~Burer and K.~M. Anstreicher.
\newblock Second-order-cone constraints for extended trust-region subproblems.
\newblock {\em SIAM Journal on Optimization}, 23(1):432--451, 2013.

\bibitem{burer2012representing}
S.~Burer and H.~Dong.
\newblock Representing quadratically constrained quadratic programs as
  generalized copositive programs.
\newblock {\em Operations Research Letters}, 40(3):203--206, 2012.

\bibitem{calafiore2004ellipsoidal}
G.~Calafiore and L.~El~Ghaoui.
\newblock Ellipsoidal bounds for uncertain linear equations and dynamical
  systems.
\newblock {\em Automatica}, 40(5):773--787, 2004.

\bibitem{casti1986linear}
J.~L. Casti.
\newblock {\em Linear dynamical systems}.
\newblock {A}cademic {P}ress {P}rofessional, {I}nc., 1986.

\bibitem{eberly20013d}
D.~H. Eberly.
\newblock 3{D} game engine design.
\newblock {\em Kaufmann, San Francisco}, 2001.

\bibitem{elzinga1974minimum}
J.~Elzinga and D.~Hearn.
\newblock The minimum sphere covering a convex polyhedron.
\newblock {\em Naval Research Logistics}, 21(4):715--718, 1974.

\bibitem{freund1985complexity}
R.M. Freund and J.B. Orlin.
\newblock On the complexity of four polyhedral set containment problems.
\newblock {\em Mathematical programming}, 33(2):139--145, 1985.

\bibitem{glineur1998pattern}
F.~Glineur.
\newblock Pattern separation via ellipsoids and conic programming.
\newblock {\em M{\'e}moire de DEA, Facult{\'e} Polytechnique de Mons, Mons,
  Belgium}, 1998.

\bibitem{gotoh2006minimal}
J.~Gotoh and H.~Konno.
\newblock Minimal ellipsoid circumscribing a polytope defined by a system of
  linear inequalities.
\newblock {\em Journal of Global Optimization}, 34(1):1--14, 2006.

\bibitem{hanasusanto2018conic}
G.~A. Hanasusanto and D.~Kuhn.
\newblock Conic programming reformulations of two-stage distributionally robust
  linear programs over wasserstein balls.
\newblock {\em Operations Research}, 66(3):849--869, 2018.

\bibitem{hanasusanto2015distributionally}
G.~A. Hanasusanto, D.~Kuhn, S.~W. Wallace, and S.~Zymler.
\newblock Distributionally robust multi-item newsvendor problems with
  multimodal demand distributions.
\newblock {\em Mathematical Programming}, 152(1-2):1--32, 2015.

\bibitem{helton2013matricial}
J.~W. Helton, I.~Klep, and S.~McCullough.
\newblock The matricial relaxation of a linear matrix inequality.
\newblock {\em Mathematical Programming}, 138(1-2):401--445, 2013.

\bibitem{henk2012lowner}
M.~Henk.
\newblock {L}{\"o}wner-{J}ohn {E}llipsoids.
\newblock {\em Documenta Mathematica}, pages 95--106, 2012.

\bibitem{horn1990matrix}
R.~A. Horn and C.~R. Johnson.
\newblock {\em Matrix analysis}.
\newblock Cambridge university press, 1990.

\bibitem{john2014extremum}
F.~John.
\newblock Extremum problems with inequalities as subsidiary conditions.
\newblock In {\em Traces and emergence of nonlinear programming}, pages
  197--215. Springer, 2014.

\bibitem{kellner2013containment}
K.~Kellner, T.~Theobald, and C.~Trabandt.
\newblock Containment problems for polytopes and spectrahedra.
\newblock {\em SIAM Journal on Optimization}, 23(2):1000--1020, 2013.

\bibitem{khachiyan1996rounding}
L.~G. Khachiyan.
\newblock Rounding of polytopes in the real number model of computation.
\newblock {\em Mathematics of Operations Research}, 21(2):307--320, 1996.

\bibitem{khachiyan1993complexity}
L.~G. Khachiyan and M.~J. Todd.
\newblock On the complexity of approximating the maximal inscribed ellipsoid
  for a polytope.
\newblock {\em Mathematical Programming}, 61(1):137--159, 1993.

\bibitem{kurzhanskiui1997ellipsoidal}
A.~B. Kurzhanski{\u\i} and I.~V{\'a}lyi.
\newblock {\em Ellipsoidal calculus for estimation and control}.
\newblock Nelson Thornes, 1997.

\bibitem{lasserre2001global}
J.~B. Lasserre.
\newblock Global optimization with polynomials and the problem of moments.
\newblock {\em SIAM {J}ournal on Optimization}, 11(3):796--817, 2001.

\bibitem{yalmip}
J.~L\"ofberg.
\newblock {YALMIP}: {A} toolbox for modeling and optimization in {MATLAB}.
\newblock In {\em IEEE International Symposium on Computer Aided Control
  Systems Design}, pages 284--289, 2004.

\bibitem{mittal2019robust}
A.~Mittal, C.~Gokalp, and G.~A. Hanasusanto.
\newblock Robust quadratic programming with mixed-integer uncertainty.
\newblock {\em Accepted in INFORMS Journal on Computing}, 2019.

\bibitem{natarajan2009constructing}
K.~Natarajan, D.~Pachamanova, and M.~Sim.
\newblock Constructing risk measures from uncertainty sets.
\newblock {\em Operations research}, 57(5):1129--1141, 2009.

\bibitem{NRZ11:mixed01}
K.~Natarajan, C.-P. Teo, and Z.~Zheng.
\newblock Mixed 0-1 linear programs under objective uncertainty: {A} completely
  positive representation.
\newblock {\em Operations Research}, 59(3):713--728, 2011.

\bibitem{parrilo2000structured}
P.~A. Parrilo.
\newblock {\em Structured semidefinite programs and semialgebraic geometry
  methods in robustness and optimization}.
\newblock PhD thesis, California Institute of Technology, 2000.

\bibitem{prasad2018improved}
M.~N. Prasad and G.~A. Hanasusanto.
\newblock Improved conic reformulations for k-means clustering.
\newblock {\em SIAM Journal on Optimization}, 28(4):3105--3126, 2018.

\bibitem{rimon1997obstacle}
E.~Rimon and S.~Boyd.
\newblock Obstacle collision detection using best ellipsoid fit.
\newblock {\em Journal of Intelligent and Robotic Systems}, 18(2):105--126,
  1997.

\bibitem{RU00:CVaR}
R.~T. Rockafellar and S.~Uryasev.
\newblock Optimization of conditional value-at-risk.
\newblock {\em Journal of Risk}, 2:21--42, 2000.

\bibitem{shapiro2002minimax}
A.~Shapiro and A.~Kleywegt.
\newblock Minimax analysis of stochastic problems.
\newblock {\em Optimization Methods and Software}, 17(3):523--542, 2002.

\bibitem{sherali2013reformulation}
H.~D. Sherali and W.~P. Adams.
\newblock {\em A reformulation-linearization technique for solving discrete and
  continuous nonconvex problems}, volume~31.
\newblock Springer Science \& Business Media, 2013.

\bibitem{silverman1980minimum}
B.~W. Silverman and D.~M. Titterington.
\newblock Minimum covering ellipses.
\newblock {\em SIAM Journal on Scientific and Statistical Computing},
  1(4):401--409, 1980.

\bibitem{sturm2003cones}
J.~F. Sturm and S.~Zhang.
\newblock On cones of nonnegative quadratic functions.
\newblock {\em Mathematics of Operations Research}, 28(2):246--267, 2003.

\bibitem{sun2004computation}
P.~Sun and R.~M. Freund.
\newblock Computation of minimum-volume covering ellipsoids.
\newblock {\em Operations Research}, 52(5):690--706, 2004.

\bibitem{todd2016minimum}
M.~J. Todd.
\newblock {\em {M}inimum-{V}olume {E}llipsoids: {T}heory and {A}lgorithms}.
\newblock SIAM, 2016.

\bibitem{xu2018copositive}
G.~Xu and S.~Burer.
\newblock A copositive approach for two-stage adjustable robust optimization
  with uncertain right-hand sides.
\newblock {\em Computational Optimization and Applications}, 70(1):33--59,
  2018.

\bibitem{yildirim2006minimum}
E.~A. Yildirim.
\newblock On the minimum volume covering ellipsoid of ellipsoids.
\newblock {\em SIAM {J}ournal on {O}ptimization}, 17(3):621--641, 2006.

\bibitem{zhen2017robust}
J.~Zhen, F.J.~De Ruiter, and D.~Den Hertog.
\newblock Robust optimization for models with uncertain {SOC} and {SDP}
  constraints.
\newblock {\em Optimization Online}, 2017.

\bibitem{zhu2009worst}
S.~Zhu and M.~Fukushima.
\newblock Worst-case conditional value-at-risk with application to robust
  portfolio management.
\newblock {\em Operations research}, 57(5):1155--1168, 2009.

\bibitem{zuluaga2006lmi}
L.~F. Zuluaga, J.~Vera, and J.~Pe{\~n}a.
\newblock {LMI} approximations for cones of positive semidefinite forms.
\newblock {\em SIAM Journal on Optimization}, 16(4):1076--1091, 2006.

\end{thebibliography}


\begin{thebibliography}{1}

\bibitem{calafiore2004ellipsoidal}
G.~Calafiore and L.~El~Ghaoui.
\newblock Ellipsoidal bounds for uncertain linear equations and dynamical
  systems.
\newblock {\em Automatica}, 40(5):773--787, 2004.

\bibitem{casti1986linear}
J.~L. Casti.
\newblock {\em Linear dynamical systems}.
\newblock {A}cademic {P}ress {P}rofessional, {I}nc., 1986.

\bibitem{kurzhanskiui1997ellipsoidal}
A.~B. Kurzhanski{\u\i} and I.~V{\'a}lyi.
\newblock {\em Ellipsoidal calculus for estimation and control}.
\newblock Nelson Thornes, 1997.

\end{thebibliography}

\newpage
\appendix

\section{\sproc}
\label{appendix:sproc}
In this section, we discuss the \mbox{ \sproc~\cite{ben2001lectures, boyd1994linear}}.
\begin{lemma}[\sproc]
	\label{lem:sproc}
	Let $\bm Q_i \in \Sbb^K, \bm q_i \in \RR^K, r_i \in \RR,\ i \in \{0\}\cup [I]$. Then the optimal value of the non-convex quadratic optimization problem 
	\begin{equation}
	\label{eq:sproc_nonconvex_quad}
	\begin{array}{clll}
	\minimize& \bm x^\top \bm Q_0 \bm x + 2\bm q_0^\top \bm x + r_0\\
	\subjectto & \bm x \in \RR^K,\\
	&\displaystyle\bm x^\top \bm Q_i \bm x + 2\bm q_i^\top \bm x + r_i \leq 0 \quad \forall i \in [I]
	\end{array}
	\end{equation}
	is $\geq 0$ if there exist $\lambda_i \geq 0\ \forall i \in [I]$ such that
	\begin{equation}
	\label{eq:sproc_suff_cond}
	\begin{bmatrix}\bm Q_0 & \bm q_0  \\ \bm q_0^\top & r_0 \end{bmatrix}
	+ \sum_{i \in [I]}\lambda_i \begin{bmatrix} \bm Q_i & \bm q_i  \\ \bm q_i^\top & r_i \end{bmatrix}\succeq \bm 0.
	\end{equation}
\end{lemma}

{

The \sproc has been used in literature to provide sufficient conditions that certify that the optimal value of a non-convex quadratic problem is non-negative \cite{ben2002robust, boyd1994linear, hanasusanto2018conic}. In the following remark, we discuss a special case when we only have linear inequalities in the optimization problem \eqref{eq:sproc_nonconvex_quad}.
\begin{remark}
	\label{rem:redundant_ellipsoid}
	In the case when all the constraints are linear, \ie, $\bm Q_i = \bm 0,\ i \in [I]$, the semidefinite constraint~\eqref{eq:sproc_suff_cond} reduces to
		\begin{equation*}
	\begin{bmatrix}\bm Q_0 & \bm q_0 + \sum_{i \in [I]}\lambda_i \bm q_i \\ \bm q_0^\top + \sum_{i \in [I]}\lambda_i \bm q_i  & r_0 + \sum_{i \in [I]}\lambda_i r_i \end{bmatrix} \succeq \bm 0,
	\end{equation*}
	which implies that $\bm Q_0 \succeq \bm 0$. Therefore, if $\bm Q_0$ is not positive semidefinite, then the sufficient conditions are never feasible; hence they do not provide any certification on the optimal value of \eqref{eq:sproc_nonconvex_quad}. We can overcome this limitation by adding a redundant quadratic constraint $\norm{\bm A\bm x+ \bm b}^2 \leq 1$ to the original problem \eqref{eq:sproc_nonconvex_quad}. Doing so does not change the optimal value of \eqref{eq:sproc_nonconvex_quad}, but the sufficient conditions \eqref{eq:sproc_suff_cond} can now be written as 
			\begin{equation*}
	\begin{bmatrix}\bm Q_0 & \bm q_0 + \sum_{i \in [I]}\lambda_i \bm q_i \\ \bm q_0^\top + \sum_{i \in [I]}\lambda_i \bm q_i  & r_0 + \sum_{i \in [I]}\lambda_i r_i \end{bmatrix} + \mu \begin{bmatrix}\bm A^2 & \bm A \bm b \\ \bm b^\top \bm A  & \bm b^\top \bm b \end{bmatrix} \succeq \bm 0.
	\end{equation*}
	Because of the additional variable $\mu$, the conditions become more flexible, and might be feasible even if $\bm Q_0$ fails to be positive semidefinite.
\end{remark}
}

\section{Alternative approaches to approximate \eqref{eq:MVE}}
\label{appendix:main}

\subsection{Scaled MVIE}
\label{appendix:smvie}

Consider the polytope $\Pc = \{ \bm x \in \RR^K : \bm S\bm x\leq \bm t \}$. It is known that the ellipsoid $\{ \bm B\bm u +\bm d : \bm u \in \RR^K,\norm{\bm u} \leq 1 \}$ with maximum volume that lies inside  $\Pc$ can be found by solving the optimization problem (see, \eg, \cite{boyd2004convex}):
\begin{equation}
\label{eq:mvie}
\begin{array}{cll}
\displaystyle \sup_{\bm B \in \Sbb^K,\; \bm d \in \RR^K} & \ \log\det (\bm B)\\
\st & \norm{\bm B \bm S_{j:}} + \bm S_{j:}^\top \bm d \leq t_j \quad  \forall j \in [J].
\end{array}
\end{equation}
Also, if $(\bm B,\bm d)$ is optimal to \eqref{eq:mvie}, then $\Esmvie = \{ K\bm B\bm u +\bm d : \norm{\bm u} \leq 1 \}$ contains $\Pc$. Therefore, $\Esmvie = \{ \bm B\bm u +\bm d : \norm{\bm u} \leq 1 \}$
if $\bm B$ and $\bm d$ are optimal in the following problem:
\begin{equation}
\label{eq:smvie}
\begin{array}{cll}
\displaystyle\sup_{\bm B \in \Sbb^K,\; \bm d \in \RR^K} & \ \log\det (\bm B)\\
\st &  \norm{\bm B \bm S_{j:}} + K \bm S_{j:}^\top \bm d \leq K t_j\quad  \forall j \in [J].
\end{array}
\end{equation}
The objective function provides the logarithm of $\vol(\Esmvie)$. The Lagrange dual of \eqref{eq:smvie} is given by
\begin{equation}
\label{eq:smvie_dual}
\begin{array}{cll}
\inf&\displaystyle K\bm \rho^\top \bm t - K - \log\det\left( - \frac{1}{2}\left(\bm S^\top\bm \Lambda  + \bm \Lambda^\top \bm S\right)\right)\\
\st& \bm \Lambda \in \RR^{J\times K},\; \bm \rho \in \RR^J,\\ 
&\bm S^\top \bm \rho = \bm 0,\\
& \norm{\bm \Lambda_{j:}} \leq \rho_j \quad \forall j \in [J].
\end{array}
\end{equation}
To show that strong duality holds, a Slater point can be constructed in the primal problem as follows. Consider a feasible solution to \eqref{eq:smvie} where $\bm B = \kappa \Ib$ and $\bm d$ is any point in the interior of $\Pc$. By choosing a sufficiently small $\kappa$, the inequalities in \eqref{eq:smvie} can be made strict. Therefore, the objective function of \eqref{eq:smvie_dual} is the logarithm of $\vol(\Esmvie)$.

\subsection{\sproc}
\label{appendix:sproc_mvep}
In this section, we use the \sproc to derive an approximation to \eqref{eq:MVE}. The constraint $\zab \leq 1$ can be written as 
$$\displaystyle \inf_{\bm x \in \Pc}\left\{-\bm x^\top \bm A^2 \bm x - 2\bm b^\top \bm A \bm x + 1 - \bm b^\top \bm b\right\} \geq 0.$$
Using Lemma \ref{lem:sproc} and the definition of $\Pc$  from \eqref{eq:quad_sets}, the above inequality is satisfied if there exist variables $\bm \mu \in \RR_+^J$ and $\lambda_i \geq 0\ \forall i \in [I]$ such that
$$-\begin{bmatrix} \bm A^2 & \bm A\bm b  \\ \bm b^\top\bm A & \bm b^\top \bm b - 1
\end{bmatrix} + \begin{bmatrix} \bm 0 & \frac{1}{2}\bm S^\top \bm\mu \\ \frac{1}{2}\bm \mu^\top\bm S & -\bm \mu^\top \bm t \end{bmatrix} + \displaystyle \sum_{i = 1}^I\lambda_i \begin{bmatrix} \bm Q_i^2 & \bm Q_i\bm q_i  \\
\bm q_i^\top\bm Q_i & \bm q_i^\top \bm q_i - 1
\end{bmatrix}\succeq \bm 0,$$
which---using the Schur complement---is satisfied if and only if
\begin{equation}\label{eq:sdp_sproc}
\begin{bmatrix} \bm 0 & \frac{1}{2}\bm S^\top \bm\mu & \bm A \\
\frac{1}{2}\bm \mu^\top\bm S & 1-\bm \mu^\top \bm t & \bm b^\top\\
\bm A & \bm b & \Ib
\end{bmatrix} + \displaystyle \sum_{i = 1}^I\lambda_i \begin{bmatrix} \bm Q_i^2 & \bm Q_i\bm q_i & \bm 0 \\
\bm q_i^\top\bm Q_i & \bm q_i^\top \bm q_i - 1& \bm 0\\	\bm 0 & \bm 0 & \bm 0
\end{bmatrix}\succeq \bm 0.
\end{equation}
Hence, by replacing the constraint $\zab \leq 1$ in \eqref{eq:MVE} with a stronger constraint \eqref{eq:sdp_sproc}, we get the following conservative approximation of \eqref{eq:MVE}:
\begin{equation*}
\label{eq:sproc_formulation}
\begin{array}{clll}
\minimize &-\log\det(\bm A)\\
\subjectto&\bm A\in \Sbb^K, \bm b \in \RR^K,\; \bm \mu \in\RR_+^J,\; \lambda_i \in \RR_+\ \forall i \in [I], \\
& \eqref{eq:sdp_sproc} \textnormal{ holds.}
\end{array}
\end{equation*}

{

\subsection{The decision rule approach of \cite{zhen2017robust}}
\label{appendix:zrh}
Consider a polytope $\Pc = \{ \bm x \in \RR^K: \bm S \bm x \leq \bm t \}$, where $\bm S \in \RR^{J\times K}$ and $\bm t \in \RR^J$. In \cite[Lemma 1]{zhen2017robust}, the authors prove that
\begin{equation}
\label{eq:uncertain_quad_cons2} \sup_{\bm x \in \Pc}\norm{\bm A \bm x + \bm b} \leq 1,
\end{equation}
if there exist $\bm V\in \RR^{J\times M}$ and $\bm v\in \RR^J$ such that the following constraints hold:
\begin{equation}
\label{eq:zrh_constraints}
\begin{array}{lll}
\displaystyle  \norm{\bm V^\top \bm t + \bm b} + \bm t^\top \bm v \leq 1, \\
\displaystyle \bm A = \bm V^\top \bm S,\\
\bm S^\top \bm v = \bm 0,\\
\norm{\bm V_{j:}} \leq v_j, \forall j\in[J].
\end{array}
\end{equation}
These sufficient conditions are based on rewriting \eqref{eq:uncertain_quad_cons2} as an equivalent two-stage robust optimization problem, and then applying a linear decision rule approximation. Substituting these conditions for the constraint $Z(\bm A, \bm b) \leq 1$ in \eqref{eq:MVE} yields the following conservative approximation:
\begin{equation}
	\label{eq:zrh_approx}
	\begin{array}{clll}
		\minimize &-\log\det(\bm A)\\
		\subjectto&\bm A\in \Sbb^K,\; \bm b \in \RR^K,\; \bm V\in \RR^{J\times K},\; \bm v\in \RR^J,\\
		&\eqref{eq:zrh_constraints} \textnormal{ holds}.
	\end{array}
\end{equation}

}

\subsection{The containment approach of \cite{kellner2013containment}}
\label{appendix:ktt}

In \cite{kellner2013containment}, the authors provide the following sufficient conditions such that a set representable as a linear matrix inequality contains another such set.
\begin{theorem}[{\cite[Theorem 4.3]{kellner2013containment}}]
	\label{thm:ktt}
	Let the set $S_Y = \{ \bm x \in \RR^K : \bm Y_0 + \sum_{k\in[K]} x_k \bm Y_k \succeq \bm 0 \}$, and the set $S_Z = \{ \bm x \in \RR^K : \bm Z_0 + \sum_{k\in[K]} x_k \bm Z_k \succeq \bm 0 \}$,  where $\bm Y_k = (Y^k_{ij}) \in \Sbb^J$ and $\bm Z_k \in \Sbb^L$ for all $k \in \{ 0 \}\cup [K]$. Then $S_Y \subseteq S_Z$ if there exist matrices $\bm C_{ij}\in \RR^{L\times L}, i,j \in [J]$, such that the following constraints hold:
	\begin{equation}
		\label{eq:ktt_cons1}
		\bm C = (\bm C_{ij})_{i,j=1}^J \succeq \bm 0, \qquad 
		\bm Z_0 \succeq \sum_{i,j=1}^J Y_{ij}^0 \bm C_{ij}, \qquad 
		\bm Z_k = \sum_{i,j=1}^J Y_{ij}^k \bm C_{ij}\;\; \forall k \in [K].
	\end{equation}
\end{theorem}

We summarize how we use this result to generate an approximation to $\Emve$. We are interested in finding conditions under which a polytope $\Pc :=  \{ \bm x \in \RR^K: \bm S \bm x \leq \bm t \} = \{ \bm x \in \RR^K:\Diag(\bm t - \bm S\bm x) \succeq \bm 0\}$ is contained in an ellipsoid $\E(\bm A,\bm b) =\{ \bm x \in \RR^K: \norm{\bm A \bm x + \bm b}^2 \leq \bm 1\} = \{ \bm x \in \RR^K: \bm F(\bm x) \succeq \bm 0\}$, where
\begin{equation*}\bm F(\bm x) = \begin{bmatrix} \Ib & \bm A\bm x + \bm b\\ (\bm A\bm x + \bm b)^\top & 1 \end{bmatrix} = \begin{bmatrix} \Ib & \bm b\\ \bm b^\top & 1 \end{bmatrix} + \sum_{k=1}^K x_k \begin{bmatrix} \bm 0 & \bm A_{k:}\\ \bm A_{k:}^\top & 0 \end{bmatrix}.\end{equation*}
Now, we can use Theorem \ref{thm:ktt} with $S_Y = \Pc$ and  $S_Z = \E(\bm A,\bm b)$ to generate constraints that ensure that $\E(\bm A,\bm b)$ contains $\Pc$. Since the matrices $\bm Y_0 = \Diag(\bm t)$ and $\bm Y_i = -\Diag(\bm S_i)$ are diagonal, the variables $\bm C_{jk}, j \neq k$ do not appear in the second and third constraints of \eqref{eq:ktt_cons1}. Therefore, we can eliminate these variables from the first constraint as well, by forcing $\bm C_{jj} \succeq \bm 0$. In light of this observation and by redefining $\bm C_{jj}$ as $\bm C_j$, we can rewrite the constraints \eqref{eq:ktt_cons1} as
\begin{equation}
\label{eq:ktt_cons2}
\bm C_j \in \Sbb_+^{K+1}\;\; \forall j \in [J], \qquad 
\begin{bmatrix} \Ib & \bm b\\ \bm b^\top & 1 \end{bmatrix} \succeq \sum_{j\in[J]} t_j \bm C_j, \qquad 
\begin{bmatrix} \bm 0 & \bm A_{k:}\\ \bm A_{k:}^\top & 0 \end{bmatrix} = \sum_{j\in[J]} -S_{jk} \bm C_j\;\; \forall k \in [K].
\end{equation}
Minimizing $-\log\det(\bm A)$ subject to the constraints in \eqref{eq:ktt_cons2} provides a conservative SDP approximation to \eqref{eq:MVE}. The elimination of these redundant variables leads to a tremendous increase in the solution speed.





\end{document}


\title{E-companion to Finding Minimum Volume Circumscribing Ellipsoids Using Generalized Copositive Programming}
\author{Areesh Mittal and Grani A.~Hanasusanto \\ Graduate Program in Operations Research and Industrial Engineering\\ The University of Texas at Austin, USA}
\date{ }
\maketitle
	
\section{Reachability in Linear Dynamical Systems}
{

	\subsection{Ellipsoidal Approximation}

	We consider a discrete-time linear dynamical system
	$$\bm x(0) = \bm 0,\quad\bm x(t+1) = \bm W_1 \bm x(t) + \bm W_2 \bm u(t) \ \forall t \in [T-1].$$
	Here $T$ is the number of time steps, $\bm x(t)\in \RR^K$ represents the state at time $t$, and $\bm u(t)\in\RR^J$ is a vector of controls selected from a convex and compact \emph{control set} $\U$. At any time, the state $\bm x(t)$ depends on the previous state and the control exercised in the previous time step. Finally, the matrices $\bm W_1 \in \RR^{K\times K}$ and $\bm W_2 \in \RR^{K\times J}$ determine the dynamics of the system. 
	
	The \emph{reachable set} $\mathcal W_T$---the set of all states reachable at time $T$---is defined as follows:
	\begin{equation*}\mathcal W_T = \{ \bm x(T) : \bm x(t+1) = \bm W_1 \bm x(t) + \bm W_2 \bm u(t) \ \forall t \in [T-1], \ \bm u(t) \in \U\ \forall t \in [T-1], \ \bm x(0) = \bm 0 \}.\end{equation*} 
	Note that the reachable set can be defined recursively as \begin{equation}\label{eq:W_recurse}\W_0 = \{ \bm 0 \}\; \textnormal{ and }\; \W_{t+1} = \bm W_1 \W_t + \bm W_2 \U = \{ \bm W_1 \bm x + \bm W_2 \bm u : \bm x \in \W_t, \bm u \in \U \} \ \forall t \in [T].\end{equation}
	A fundamental problem in the study of linear dynamical systems involves verifying whether a state is reachable at time $T$ \cite[Chapter 3]{casti1986linear}. For a single state, reachability can be determined by solving a convex feasibility problem whose size is quadratic in the number of time steps $T$ (since both the number of variables and constraints are proportional to $T$). 
	
	When $T$ grows large, determining the reachability of multiple states becomes computationally expensive (since the feasibility problem must be solved separately for every state). For this reason, several authors have proposed outer ellipsoidal approximations for $\W_T$ \cite{calafiore2004ellipsoidal, kurzhanskiui1997ellipsoidal}. A state is then deemed reachable if it falls within the approximating ellipsoid. Checking this condition is much faster than solving the convex feasibility problem. The reduction in time, however, can lead to false positives (states which lie inside the ellipsoid but are not reachable). Thus, generating a tight ellipsoidal approximation is desirable.
	
	The previous studies describe ellipsoidal approximations for $\W_T$ only when the control set $\U$ is an ellipsoid. In contrast, in this section, we consider the case when the control set $\U= \{ \bm u \in \RR^J : \bm S \bm u \leq \bm t \}$ is a polytope, where $\bm S \in \RR^{M\times J}$ and $\bm t \in \RR^M$.

	We find the ellipsoidal approximation for $\mathcal W_T$ as follows. Let $\E(\bm A_t, \bm b_t)$ be the ellipsoid found at time $t$ that contains $\W_t$. First, note that $\W_1 = \bm W_2 \U = \{ \bm W_2 \bm u : \bm S \bm u \leq \bm t \}$, which is a polytope. Therefore, we can use Theorem \ref{thm:C0_polytope} to find an ellipsoid $\E(\bm A_1, \bm b_1)$ that contains $\W_1$. Next, at each step, we use the ellipsoid $\E(\bm A_t, \bm b_t)$ that contains $\W_t$ to generate the ellipsoid $\E(\bm A_{t+1}, \bm b_{t+1})$ that contains $\W_{t+1}$. To find $\E(\bm A_{t+1}, \bm b_{t+1})$ from $\E(\bm A_t, \bm b_t)$, we solve the following optimization problem: 
	\begin{equation}
	\label{eq:Etplus1}
	\begin{array}{clll}
	\minimize &-\log\det(\bm A)\\
	\subjectto&\bm A\in \Sbb^K,\; \bm b \in \RR^K,\; \bm N \in \RR_+^{J\times J},\; \bm F\in \Sbb^K,\; \bm g\in \RR^K,\; h\in \RR,\gap \\
	& \begin{bmatrix} \bm F & \bm g \\ \bm g^\top & h - 1
	\end{bmatrix} \displaystyle \preceq - \Sb^\top\bm N\Sb + \lambda \bm J_t,\gap \\
	&\begin{bmatrix}\bm F &\bm g & (\bm A\overline{\bm W})^\top\\ \bm g^\top & h& \bm b^\top \\ \bm A\overline{\bm W}& \bm b & \Ib \end{bmatrix} \succeq \bm 0,
	\end{array}
	\end{equation}
			where
	\begin{equation*}
	\begin{array}{clll}
	\displaystyle\overline{\bm W} = \begin{bmatrix} \bm W_1 & \bm W_2	\end{bmatrix},\;\;  \Sb = \begin{bmatrix} \bm 0 & -\bm S &\bm t \end{bmatrix} \in \RR^{M\times(K+J+1)},\; \text{ and }\; \bm J_t = \begin{bmatrix}  \bm A_t^2 & \bm 0 & \bm A_t^\top \bm b_t \\ \bm 0 & \bm 0 & \bm 0 \\ \bm b_t^\top \bm A_t & \bm 0 & \bm b_t^\top \bm b_t - 1 \end{bmatrix} \in \Sbb^{K+J+1}.
	\end{array}
	\end{equation*}
	If $(\bm A^\ast, \bm b^\ast, \bm N^\ast, \bm F^\ast, \bm g^\ast, h^\ast)$ is an optimal solution to \eqref{eq:Etplus1}, then we use  $\E(\bm A_{t+1}, \bm b_{t+1}) = \E(\bm A^\ast, \bm b^\ast)$ as the ellipsoid for next time step. In Proposition \ref{prop:reachability_Et}, presented below, we prove that $\E(\bm A^\ast, \bm b^\ast)$ indeed contains $\W_{t+1}$. Thus, repeating this procedure, we obtain the ellipsoidal approximation $\E(\bm A_T, \bm b_T)$ that contains $\W_T$.  In Figure \ref{fig:reachability}, we depict the evolution of $\E_T$ and $\W_T$ as $T$ is increased for the case when $J = K = 2,$
$$\bm W_1 = \begin{bmatrix} 0.9202 & -0.0396 \\ 0.0777 & 0.9800 \end{bmatrix}\quad \bm W_1 = \Ib, \quad \textnormal{and}\quad \U= \{ \bm u  \in \RR^2: -\ee \leq \bm u \leq \ee, \norm{\bm u}_1 \leq 1.4 \}.$$ 
	\begin{proposition}
		\label{prop:reachability_Et}
		Let $\E(\bm A_t, \bm b_t) = \{\bm x \in \RR^K : \norm{\bm A_t \bm x + \bm b_t} \leq 1 \}$ be the ellipsoid containing $\W_t$.  If the matrix $\bm A \in \Sbb_{++}^K$ and the vector $\bm b\in \RR^K$ satisfy the constraints of \eqref{eq:Etplus1}, then the ellipsoid $\E(\bm A, \bm b)$ contains $\W_{t+1}$.
	\end{proposition}
	\begin{proof}
		See Section 2.1 of the electronic companion.
	\end{proof}
\begin{figure}
	\begin{center}
		\begin{subfigure}{.22\textwidth}
			\includegraphics[width=\linewidth]{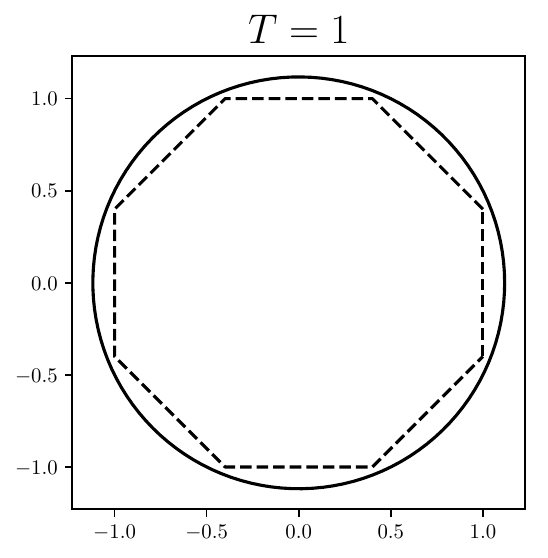}
		\end{subfigure}\hfill%
		\begin{subfigure}{.22\textwidth}
			\includegraphics[width=\linewidth]{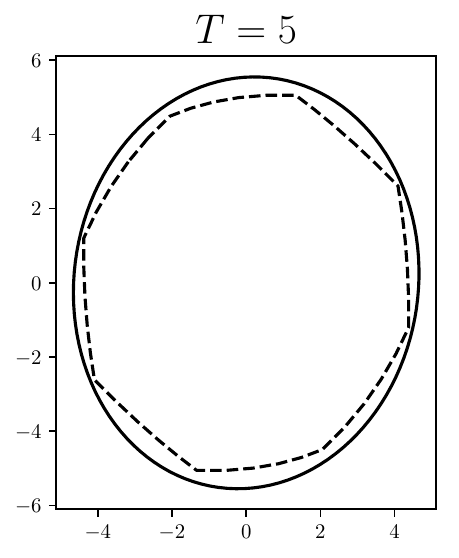}
		\end{subfigure}\hfill%
		\begin{subfigure}{.22\textwidth}
			\includegraphics[width=\linewidth]{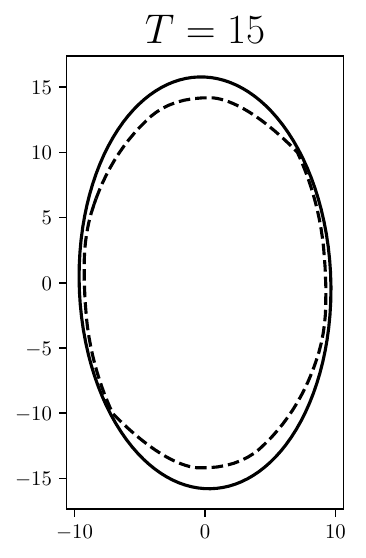}
		\end{subfigure}\hfill%
		\begin{subfigure}{.22\textwidth}
			\includegraphics[width=\linewidth]{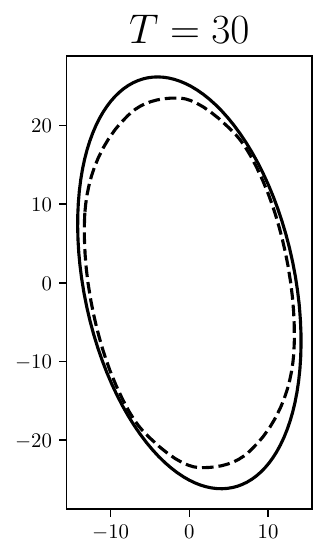}
		\end{subfigure}\hfill%
		\caption[Reachability in Linear Dynamical Systems]{Reachability in Linear Dynamical Systems. For different values of $T$, the dotted curves depict the boundary of the actual reachable set $\mathcal W_T$, and the solid curves depict the ellipsoidal approximation $\E_T$ generated by our method.}
		\vspace{-.2in}
		\label{fig:reachability}
	\end{center}
\end{figure}
Note that the time needed to generate ellipsoidal approximation to $\W_T$ is linear in $T$, since we solve $T$ semidefinite programs whose size does not depend on $T$. This is an improvement from checking the reachability exactly, where the size of the problems is quadratic in $T$. We now provide numerical experiments to demonstrate the effectiveness of our approach on randomized instances of linear dynamical problems.

\subsection{Numerical Experiments}
\label{sec:reachability_experiments}

For our experiments, we consider the control set $\U= \{ \bm u  \in \RR^K: -\ee \leq \bm u \leq \ee, \norm{\bm u}_1 \leq \sqrt K \}$, $T = 30$ time steps, $J = K$ and $\bm W_2 = \Ib$. We randomly generate $\bm W_1$ such that the maximum eigenvalue is less than $1$ and the condition number is at most $2$, which ensures that the system is stable. For given lengths of the state vector $K$, this is achieved with the following MATLAB code:
\begin{lstlisting}[language = Matlab]
      function W1 = generate_W1(K)
          max_eigenvalue = 1;
          condition_number = 1 + rand(); % choose randomly between 1 and 2
          X = randn(K);
          [U, S, V] = svd(X);
          singular_values = linspace(condition_number, 1, K);
          W1 = U * diag(singular_values) * V'; % update the singular values
          W1 = max_eigenvalue * W1 / max(abs(eig(W1)));
      end
\end{lstlisting}

Next, we explain the metric that we use to compare the volumes of the sets  $\E_T$ and $\W_T$. For a convex and compact set $S \subseteq \RR^K$, let $\gamma(S) = \left(\prod_{k=1}^K \rho_k(S)\right)^{1/K}$, where $\rho_k(S) = \max\{ x_k : \bm x \in S \} - \min\{ x_k : \bm x \in S \}, \; k \in [K].$ Here, $\rho_k(S)$ represents the length of the projection of $S$ onto the $k^{\textnormal{th}}$ coordinate, and $\gamma(S)$ represents the geometric mean of these lengths. We use the parameter $\gamma = \gamma(\E_T) / \gamma(\W_T)$ to measure the closeness of the sizes of $\E_T$ and $\W_T$. We can see that $\gamma \geq 1$ since $\W_T \subseteq \E_T$, and the closer $\gamma$ is to one, the better $\E_T$ approximates $\W_T$. As an example, in the case depicted in Figure \ref{fig:reachability}, $\gamma = 1.06$ at time step $T = 30$.

\begin{table}
	\centering
	\begin{tabular}{c|c|c|c}
		$K$  &  $\gamma$ & MVE Time (s) & Exact Time (s) \\
		[0.0mm] \hline
		$2$  & 1.18 & 0.13 & 52.3\\
		$5$  & 1.24 & 0.48 & 87.1\\
		$10$ & 1.30 & 3.2 &  205\\
		$20$ & 1.39 & 55.5 & 325\\
		\hline\hline
	\end{tabular}
	\caption[Relative suboptimality for minimum volume ellipsoids]{ Reachability in Linear Dynamical Systems: The quality of the approximation of $\E_T$ ($\gamma$), the time required to generate $\E_T$ and check the membership of $10,000$ points in $\E_T$ (``MVE Time''), and the time required to check the membership of the same points in $\W_T$ (``Exact Time'') for different lengths of the state vector. \label{table:reachability_results}}
\end{table}

For different values of $K$, we perform the experiment on $100$ randomly generated instances. For each instance, we record the time that it takes to check the membership within $\E_T$ and $\W_T$ of $10,000$ points sampled randomly from $[-10, 10]^K$. In Table \ref{table:reachability_results}, we present \begin{inparaenum}[1)] \item average solution quality ($\gamma$), \item the total computational time required for generating $\E_T$ and then checking the membership of the randomly sampled points inside $\E_T$, and \item the computational time required to establish the membership of these points in $\W_T$. \end{inparaenum} We can observe that, for $K=20$, $\E_T$ is about $39\%$ bigger in size than $\W_T$. However, the computational time needed to find $\E_T$ and then checking the membership inside it is much less than checking the feasibility of these points exactly.
}

\section{Proofs}

\subsection{Propositions}
\label{appendix:dro_prop_proof}

\noindent \textit{Proof of Proposition \ref{prop:sproc_improvement}: }  For the set \mbox{$\{\bm x \in \RR^K : \bm S\bm x \leq \bm t, \norm{\bm Q \bm x + \bm q}^2 \leq 1 \}$}, then $\Esproc = \E(\bm A, \bm b)$, where $\bm A$ and $\bm b$ are optimal in the following optimization problem (see Appendix \ref{appendix:sproc_mvep}):
\begin{equation}
\label{eq:sproc_1ellip}
\begin{array}{clll}
\minimize &-\log\det(\bm A)\\
\subjectto&\bm A\in \Sbb^K, \bm b \in \RR^K,\; \lambda \in \RR_+ ,\; \bm \mu  \in\RR_+^J, \vspace{.1cm} \\
&\begin{bmatrix} \bm 0 & \frac{1}{2}\bm S^\top \bm\mu & \bm A \\
\frac{1}{2}\bm \mu^\top\bm S & 1-\bm \mu^\top \bm t & \bm b^\top\\
\bm A & \bm b & \Ib
\end{bmatrix} + \lambda \begin{bmatrix} \bm Q^2 & \bm Q\bm q & \bm 0 \\
\bm q^\top \bm Q & \bm q^\top \bm q - 1& \bm 0\\ \bm 0 & \bm 0 & \bm 0
\end{bmatrix}\succeq \bm 0.
\end{array}
\end{equation}
The dual of \eqref{eq:sproc_1ellip} can be written as
\begin{equation}
\label{eq:sproc_dual}
\begin{array}{clll}
\maximize & K + \log\det(-2\bm F) - \kappa - \tr(\bm \Gamma)\\
\subjectto &\bm \Omega \in \Sbb^K,\; \bm \xi \in\RR^K,\; \bm F \in \Sbb^K,\; \bm \Gamma \in \Sbb^K, \kappa \in \RR,\\
&\bm S\bm \xi \leq \kappa \bm t,\vspace{.1cm}\\
&\tr\left(\begin{bmatrix} \bm \Omega & \bm \xi \\ \bm \xi^\top & \kappa \\\end{bmatrix} \begin{bmatrix} \bm Q^2 & \bm Q\bm q \\\bm q^\top\bm Q & \bm q^\top \bm q - 1\end{bmatrix}\right) \leq 0, \vspace{.1cm}\\
&\begin{bmatrix} \bm \Omega & \bm \xi & \bm F \\ \bm \xi^\top & \kappa & \bm 0 \\ \bm F & \bm 0 & \bm \Gamma
\end{bmatrix} \succeq \bm 0.
\end{array}
\end{equation}
To prove the theorem, we construct a pair of primal and dual feasible solutions which generate the same objective function values to their respective problems. To this end,
consider the following solution to the primal problem \eqref{eq:sproc_1ellip}:
$$\bm A = \bm Q,\; \bm b = \bm q,\; \bm \mu = \bm 0,\; \lambda = 1.$$
This solution is feasible since $\bm \mu \geq \bm 0$, $\lambda \geq 0$ and 
\begin{equation*}
\begin{array}{rlll}
\begin{bmatrix} \bm 0 & \frac{1}{2}\bm S^\top \bm\mu & \bm A \\
\frac{1}{2}\bm \mu^\top\bm S & 1-\bm \mu^\top \bm t & \bm b^\top\\ \bm A & \bm b & \Ib \end{bmatrix} + \lambda \begin{bmatrix} \bm Q^2 & \bm Q\bm q & \bm 0 \\ \bm q^\top \bm Q & \bm q^\top \bm q - 1& \bm 0\\	\bm 0 & \bm 0 & \bm 0 \end{bmatrix} 
= \begin{bmatrix} \bm Q^2 & \bm Q\bm q & \bm Q \\ 	\bm q^\top \bm Q & \bm q^\top \bm q & \bm q^\top\\ \bm Q & \bm q & \Ib \end{bmatrix}
= \begin{bmatrix}\bm Q\\ \bm q^\top \\ \Ib \end{bmatrix}\begin{bmatrix}\bm Q\\ \bm q^\top \\ \Ib \end{bmatrix}^\top \succeq \bm 0.
\end{array}
\end{equation*}
Next, consider the following solution to the dual problem \eqref{eq:sproc_dual}:
$$\bm F = -\frac{1}{2}\bm Q^{-1},\; \bm \Gamma = \frac{\Ib}{2},\; \kappa = \frac{K}{2},\; \bm \xi = \kappa \bm x_c,\; \bm \Omega = \kappa \bm x_c \bm x_c^\top + 2 \bm F^2,$$
where $\bm x_c = -\bm Q^{-1}\bm q$ is the center of the ellipsoid $\E$. We claim that this solution is feasible to \eqref{eq:sproc_dual}. Under the assumption that the center of the ellipsoid lies inside the polytope, we get that $\bm S \bm x_c \leq \bm t$, which implies that $\bm S \bm \xi \leq \kappa \bm t$. Next, we have that
$$\begin{bmatrix} \bm \Omega & \bm \xi & \bm F \\ \bm \xi^\top & \kappa & \bm 0 \\	\bm F & \bm 0 & \bm \Gamma \end{bmatrix} =
\begin{bmatrix} \kappa \bm x_c \bm x_c^\top + 2 \bm F^2 & \kappa \bm x_c & \bm F \\ \kappa \bm x_c^\top & \kappa& \bm 0 \\	\bm F & \bm 0 & \frac{\Ib}{2} \end{bmatrix}
= \kappa \begin{bmatrix}\bm x_c \\ 1 \\ \bm 0\end{bmatrix}\begin{bmatrix}\bm x_c \\ 1 \\ \bm 0\end{bmatrix}^\top + \frac{1}{2}\begin{bmatrix}2 \bm F \\ \bm 0 \\ \Ib\end{bmatrix} \begin{bmatrix}2 \bm F \\ \bm 0 \\ \Ib \end{bmatrix}^\top \succeq \bm 0, $$
where the last inequality uses the fact that $\kappa = K/2 \geq 0$. Also,
\begin{equation*}
\begin{array}{rll}
\tr\left(\begin{bmatrix} \bm \Omega & \bm \xi \\ \bm \xi^\top & \kappa \end{bmatrix} \begin{bmatrix} \bm Q^2 & \bm Q\bm q \\\bm q^\top\bm Q & \bm q^\top \bm q - 1\end{bmatrix}\right) & = \tr(\bm \Omega\bm Q^2) + 2 \bm \xi^\top \bm Q \bm q + \kappa(\bm q^\top \bm q - 1)\\
& \displaystyle = \tr((\kappa \bm x_c \bm x_c^\top + 2 \bm F^2)\bm Q^2) + 2 \kappa \bm x_c^\top \bm Q \bm q + \kappa(\bm q^\top \bm q - 1)\\
&\displaystyle = \kappa\bm x_c^\top \bm Q^2 \bm x_c + \frac{1}{2}\tr(\Ib) + 2 \kappa \bm x_c^\top \bm Q \bm q + \kappa(\bm q^\top \bm q - 1)\gap \\
&\displaystyle = \frac{1}{2}\tr(\Ib) - \kappa = \frac{K}{2} - \frac{K}{2} = 0.
\end{array}
\end{equation*}
Therefore, all constraints in the dual problem are satisfied. Finally, both of these solutions give an objective function value of $-\log\det(\bm Q)$. Thus, $\bm A = \bm Q$ and $\bm b = \bm q$ is an optimal solution to the primal problem, which implies that $\Esproc = \E(\bm A,\bm b) = \E(\bm Q,\bm q)$. Furthermore, the solution is unique because the feasible region is convex and the objective function $-\log\det(\bm A)$ is strictly convex in the space of positive definite matrices. \qed

\vspace{2mm}

{
	\noindent \textit{Proof of Proposition \ref{prop:compare_zrh_smvie}: } To prove this result, we demonstrate that any feasible solution in \eqref{eq:zrh_approx} can be used to construct a feasible solution to \eqref{eq:smvie_dual} with the same or lower objective value. To this end, consider a feasible solution $(\bm A, \bm b, \bm V, \bm v)$ of \eqref{eq:zrh_approx}. Now consider the following solution to \eqref{eq:smvie_dual}: $$\bm \Lambda = - \bm V,\; \bm \rho = \bm v.$$ 
	The last two constraints of $\eqref{eq:zrh_approx}$ imply that the constraints of \eqref{eq:smvie} are satisfied. Also note that
	\begin{equation}
	\begin{array}{rllll}
	-\log\det(\bm A) & \displaystyle = \log\det\left(\frac{1}{2}(\bm V^\top \bm S + \bm S^\top \bm V)\right) \vspace{1mm}\\
	& \displaystyle\geq \log\det\left(\frac{1}{2}(\bm V^\top \bm S + \bm S^\top \bm V)\right) + K(\bm t^\top \bm v - 1)\vspace{1mm}\\
	& \displaystyle = \log\det\left(-\frac{1}{2}(\bm \Lambda^\top \bm S + \bm S^\top \bm \Lambda)\right) + K(\bm t^\top \bm \rho - 1).
	\end{array}
	\end{equation}
	Here, the first equality follows since $\bm A$ is symmetric, which implies that $\bm A = (\bm A + \bm A^\top) / 2 = (\bm V^\top \bm S + \bm S^\top \bm V) / 2$, and the inequality follows from the first constraint of \eqref{eq:zrh_approx} since $\norm{\bm V^\top \bm t + \bm b} \geq 0$. Therefore, the result holds.\qed
}

\vspace{2mm}

\noindent \textit{Proof of Proposition \ref{prop:dro}: } The PLD restriction of \eqref{eq:dro_semiinf} can be written as follows:
\begin{equation}
\label{eq:dro_pldr2}
\begin{array}{clll}
\displaystyle \inf_{\bm x, \bm y(\cdot)} & \displaystyle \bm c^\top \bm x + \sup_{\QQ\in\Qc} \EE_\QQ [(\bm D\tilde{\bm\xi}+ \bm d)^\top \bm y(\tilde{\bm\xi})] ,\\
\st & \bm x \in \X,\vspace{1mm} \\
& \bm T_\ell(\bm x)^\top\bm \xi + h_\ell(\bm x) \leq \bm (\bm W_\ell\bm\xi+\bm w_\ell)^\top\bm (\bm Y_j\bm\xi+\bm y_j)\quad \forall \bm\xi \in \Xi_j\; \forall j \in [J]\; \forall \ell \in [L].
\end{array}
\end{equation}
First, for the objective function, observe that
\begin{equation*}
\begin{array}{rcll}
\displaystyle \sup_{\QQ\in\Qc} \EE_\QQ [(\bm D\tilde{\bm\xi}+ \bm d)^\top \bm y(\tilde{\bm\xi})] = & \displaystyle \sup_{\nu(\cdot)\geq 0} & \displaystyle \sum_{j\in [J]}\int_{\bm\xi\in \Xi_j} (\bm D\bm\xi+ \bm d)^\top (\bm Y_j\bm\xi+\bm y_j) \nu(d\bm\xi) \vspace{1mm} \\
&\st &\displaystyle \sum_{j\in [J]}\int_{\bm\xi\in \Xi_j} \nu(d\bm\xi) = 1,\vspace{1mm}\\
&&\displaystyle \sum_{j\in [J]}\int_{\bm\xi\in \Xi_j} \bm\xi \nu(d\bm\xi) = \bm \mu,\vspace{1mm}\\
&&\displaystyle \sum_{j\in [J]}\int_{\bm\xi\in \Xi_j} \bm\xi \bm\xi^\top \nu(d\bm\xi) \preceq \bm \Sigma.
\end{array}
\end{equation*}
By weak duality, we get that 
\begin{equation*}
\begin{array}{rlll}
\displaystyle \sup_{\QQ\in\Qc} \EE_\QQ [(\bm D\tilde{\bm\xi}+ \bm d)^\top \bm y(\tilde{\bm\xi})] \leq 
\displaystyle \inf & \displaystyle \alpha + \bm\beta^\top \bm\mu + \tr(\bm\Gamma\bm \Sigma)  \vspace{1mm} \\
\st &\alpha \in \RR,\bm\beta\in \RR^K,\bm\Gamma \in \Sbb_+^K, \vspace{1mm}\\
&\displaystyle \alpha + \bm\beta^\top\bm\xi + \bm\xi^\top \bm\Gamma\bm\xi \geq (\bm D\bm\xi+ \bm d)^\top (\bm Y_j\bm\xi+\bm y_j)\quad \forall \bm\xi\in \Xi_j\; \forall  j \in [J].
\end{array}
\end{equation*}
The constraint of the optimization problem above holds if and only if, for all $j \in [J]$,  the optimal value of the problem
\begin{equation*}
\begin{array}{rlll}
\displaystyle \inf_{\bm\xi\in \RR^K} & \alpha + \bm\beta^\top\bm\xi + \bm\xi^\top \bm\Gamma\bm\xi - (\bm D\bm\xi+ \bm d)^\top (\bm Y_j\bm\xi+\bm y_j) \\
\st & \bm S_j\bm\xi\leq \bm t_j,\quad \norm{\bm A_j\bm \xi + \bm b_j}^2 \leq 1,
\end{array}
\end{equation*}
is $\geq 0$. Next, using the \sproc (Lemma \ref{lem:sproc}), we get that this constraint holds if the first semidefinite constraint of \eqref{eq:dro_pldr} holds. Therefore, replacing the former by the latter, we get an upper bound on the optimal decision rules problem. Similarly, the final constraint of \eqref{eq:dro_pldr2} is equivalent to the constraint that, for all $j\in [J]$ and $\ell \in [L]$, the optimal value of the following optimization problem is greater than or equal to 0:
\begin{equation*}
\begin{array}{rlll}
\displaystyle \inf_{\bm\xi\in \RR^K} & \bm (\bm W_\ell\bm\xi+\bm w_\ell)^\top\bm (\bm Y_j\bm\xi+\bm y_j) -  \bm T_\ell(\bm x)^\top\bm \xi + h_\ell(\bm x) \\
\st & \bm S_j\bm\xi\leq \bm t_j, \quad \norm{\bm A_j\bm \xi + \bm b_j}^2 \leq 1.
\end{array}
\end{equation*}
Using Lemma \ref{lem:sproc}, we get that the above constraint holds if the second semidefinite constraint of \eqref{eq:dro_pldr} holds. Therefore, the SDP \eqref{eq:dro_pldr} provides a feasible decision rule approximation, and the optimal value of \eqref{eq:dro_pldr} provides an upper bound to the optimal value of the DRO model \eqref{eq:dro_semiinf}. \qed

\vspace{2mm}

{
	
	\noindent \textit{Proof of Proposition \ref{prop:reachability_Et}: } Consider the set $\overline{\W} = \bm W_1 \E(\bm A_t, \bm b_t) + \bm W_2 \U =  \{ \bm W_1 \bm x + \bm W_2 \bm u: \bm S\bm u \leq \bm t,\; \norm{\bm A_t \bm x + \bm b_t}^2 \leq 1 \}$. Since $\W_t \subseteq \E(\bm A_t, \bm b_t)$, we get that $\W_{t+1} = \bm W_1 \W_t + \bm W_2 \U \subseteq \bm W_1 \E(\bm A_t, \bm b_t) + \bm W_2 \U = \overline{\W}$. Therefore, $\W_{t+1}\subseteq \overline{\W}$. Next, we find sufficient conditions that imply that an ellipsoid contains $\overline{\W}$, which, in turn, implies that the ellipsoid contains $\W_{t+1}$. To this end, note that $\overline{\W}$ is an special case of the setup in Remark~\ref{remark:affine_mapping} with $\Pc = \{ (\bm x, \bm u) \in \RR^{K+J}: \bm S\bm u \leq \bm t,\; \norm{\bm A_t \bm x + \bm b_t}^2 \leq 1 \}, \;\bm C = \overline{\bm W}$ and $\bm d = \bm 0$. Therefore, the ellipsoid $\E(\bm A_{t+1},\bm b_{t+1})$ contains $\overline{\W}$ if and only if there exist $\bm F\in \Sbb^{K+J},\; \bm g\in \RR^{K+J},\; h\in \RR$ such that
	\begin{equation*}
	\begin{array}{clll}
	\begin{bmatrix} \bm F  & \bm g \\ \bm g^\top  & h - 1 \end{bmatrix} \preceq_{\C(\K)} \bm 0 \quad \text{and} \quad \begin{bmatrix}\bm F &\bm g & (\bm A_{t+1}\overline{\bm W})^\top\\ \bm g^\top & h& \bm b_{t+1}^\top \\ \bm A_{t+1}\overline{\bm W}& \bm b_{t+1} & \Ib \end{bmatrix} \succeq \bm 0,
	\end{array}
	\end{equation*}
	where $\K = \{ (\bm x, \bm u, \tau ) \in \RR^{K+J+1} : \bm S\bm u \leq \tau \bm t,\; \norm{\bm A_t \bm x + \tau \bm b_t}^2 \leq \tau^2 \}$. Since the set $\Pc$ is defined by linear and quadratic inequalities, using Theorem \ref{thm:C0_quad}, we get that the constraints of $\eqref{eq:Etplus1}$ imply that $\E(\bm A_{t+1}, \bm b_{t+1})$ contains $\overline{\W}$. Hence, $\E(\bm A_{t+1}, \bm b_{t+1})$ contains $\W_{t+1}$. \qed
}

\subsection{Claims in Example \ref{example:chipped}}
\label{appendix:chipped}

In this section, we prove that in Example \ref{example:chipped}, $R_{\textnormal{cop}}=  O(K^{1/4})$, and $R_{\textnormal{smvie}} = \Theta(\sqrt{K})$.
Consider the solution  
$$\bm A = k_1 \bm \Ib + k_2 \ee \ee^\top,\; \bm b = k_3 \ee, \; \bm N = \begin{bmatrix}
\bm 0 & k_4 \Ib & \bm 0 \\ k_4 \Ib & \bm 0 & k_5 \ee \\ \bm 0 & k_5 \ee^\top & 0
\end{bmatrix},$$
$$\textnormal{where,}\quad k_1 = \sqrt{\frac{K^2 - 1}{(\sqrt{K}-1)K^2}}, \; k_2 = \frac{1}{K}\left(1 + \frac{1}{K} - k_1 \right),\; k_3 = -\frac{1}{\sqrt{K}},\; k_4 = \frac{k_1^2}{2},\; k_5 = k_2\left( k_1 + \frac{K}{2}k_2 \right).$$
It can be checked that this solution is feasible to \eqref{eq:MVE_Ab}. Therefore, $R_{\textnormal{cop}} \leq (1/\det(\bm A))^{1/K}$. The eigenvalues of $\bm A$ are $k_1 + K k_2$, and $k_1$ with a multiplicity of $K-1$. Therefore,
$$R_{\textnormal{cop}} \leq (1/\det(\bm A))^{1/K} =  \left(\left(k_1 + K k_2 \right)k_1^{K-1}\right)^{-\frac{1}{K}} = \left(\left(1+\frac{1}{K}\right)k_1^{K-1}\right)^{-\frac{1}{K}} = O\left(\frac{1}{k_1}\right) =  O\left(K^{\frac{1}{4}} \right).$$

Next, for $\Esmvie$, consider the following solution to the primal problem \eqref{eq:smvie}: $\bm B = m_1 \Ib + m_2 \ee\ee^\top,\; \bm d= m_3 \ee,$ and the following solution to the dual problem \eqref{eq:smvie_dual}: $\bm \rho = 1/\sqrt{K} \begin{bmatrix} \bm 0;\; \ee ;\; 1\end{bmatrix},\; \bm \Lambda = \begin{bmatrix} \bm 0;\; m_4 \Ib + m_5 \ee\ee^\top;\; m_6 \ee^\top  \end{bmatrix},$ where
$$m_1 = \frac{K}{\sqrt{K+1}},\; m_2 = \frac{1}{K+1} - \frac{1}{\sqrt{K+1}},\; m_3 = \frac{\sqrt{K}}{K+1},\;m_4 = \frac{\sqrt{K+1}}{K},\; m_5 = \frac{1-\sqrt{K+1}}{K^2},\;m_6 = -\frac{1}{K}.$$
It can be verified that these solutions have the same objective function value, and are feasible---and therefore optimal---to their respective problems. The eigenvalues of $\bm B$ are $m_1 + K m_2$, and $m_1$ with a multiplicity of $K-1$. Therefore
$$\det(\bm B) = (m_1+K m_2)m_1^{K-1} = \frac{K}{K+1}\left(\frac{K}{\sqrt{K+1}} \right)^{K-1} = \frac{K^K}{(\sqrt{K+1})^{K+1}},$$
which implies that
$$\lim_{K\rightarrow \infty}\frac{R_{\textnormal{smvie}}}{\sqrt{K}} = \lim_{K\rightarrow \infty}\frac{\det(\bm B)^{1/K}}{\sqrt{K}} = 1.$$
Therefore, $R_{\textnormal{smvie}} = \Theta (\sqrt{K}).$

\subsection{Claims in Example \ref{example:dro2}}
\label{appendix:dro_ex2}

{Using \sproc (Lemma \ref{lem:sproc}), we substitute the semidefinite programming approximation for the inequalities of \eqref{eq:dro_ex2}, yielding the following approximation to \eqref{eq:dro_ex2}:}
\begin{equation}
\label{eq:dro_ex_primal}
\begin{array}{clll}
\inf & x\\
\st & x\in \RR,\; \bm Y\in \RR^{K\times K},\; \bm y \in \RR^K,\; \bm \rho_1 \geq \bm 0,\; \bm \rho_2 \geq \bm 0,\; \lambda_1 \geq 0,\; \lambda_2 \geq 0,\vspace{1mm} \\
&\displaystyle  \begin{bmatrix}\frac{1}{2}(\bm Y+\bm Y^\top) & \frac{1}{2}(\bm S^\top \bm \rho_1 + \bm Y^\top \ee + \bm y ) \\ \frac{1}{2}(\bm S^\top \bm \rho_1 + \bm Y^\top \ee + \bm y )^\top & -1 + \ee^\top \bm y - \bm t^\top \bm \rho_1 \end{bmatrix}+ \lambda_1 \bm J(s)\succeq \bm 0,\vspace{1mm}\\
&\displaystyle  \begin{bmatrix}-\frac{1}{2}(\bm Y+\bm Y^\top) & \frac{1}{2}(\bm S^\top \bm \rho_2 + \bm Y^\top \ee - \bm y ) \\ \frac{1}{2}(\bm S^\top \bm \rho_2 - \bm Y^\top \ee - \bm y )^\top & x - \ee^\top \bm y - \bm t^\top \bm \rho_2 \end{bmatrix}  + \lambda_2 \bm J(s) \succeq \bm 0,
\end{array}
\end{equation}
where $$J(s) = \frac{1}{K}\begin{bmatrix} 4\Ib & -2\ee \\ -2\ee^\top & -K s \end{bmatrix}.$$ 
The dual of the SDP \eqref{eq:dro_ex_primal} is given by 
\begin{equation*}
\begin{array}{rlll}
\sup & h_1 \\
\st & \bm F_1,\bm F_2 \in \Sbb^{K},\;\; \bm g_1,\bm g_2 \in \RR^{K},\;\; h_1,h_2 \in \RR \\
& \begin{bmatrix} \bm F_1 & \bm g_1\\ \bm g_1^\top & h_1 \end{bmatrix} \succeq \bm 0, \quad \begin{bmatrix} \bm F_2 & \bm g_2\\ \bm g_2^\top & h_2 \end{bmatrix} \succeq \bm 0,\vspace{1mm}\\
& \tr\left(\bm J(s)\begin{bmatrix} \bm F_1 & \bm g_1\\ \bm g_1^\top & h_1 \end{bmatrix}\right) \leq 0, \quad \tr\left(\bm J(s) \begin{bmatrix} \bm F_2 & \bm g_2\\ \bm g_2^\top & h_2 \end{bmatrix} \right) \leq 0,\\
& \bm g_1 - \bm g_2 = (h_2-h_1)\ee,\\
& \bm F_1 - \bm F_2 = \ee(\bm g_2 - \bm g_1)^\top, \\
& \bm S\bm g_1 \leq h_1 \bm t,\;\bm S\bm g_2 \leq h_2 \bm t,\\ 
& h_2 = 1.
\end{array}
\end{equation*}
Consider the following cases:
\begin{itemize}
	\item Case 1: $0\leq s \leq 2$. Consider the following solution to the dual problem:
	$$\bm F_2 = \frac{1}{4}\ee\ee^\top ,\; \bm g_2 = \frac{1}{2}\ee,\; h_2 = 1,\; h_1 = \frac{9}{8-s},\; \bm F_1 = \bm F_2 + \frac{1+s}{8-s}\ee\ee^\top,\; \bm g_1 = \bm g_2 - \frac{1+s}{8-s}\ee,$$
	and the following solution to the primal problem:
	$$\lambda_2 = 0,\; \lambda_1 = \frac{1}{8-s},\; \bm\rho_1 = \bm 0,\; \bm\rho_2 = \bm 0,\; \bm y = \frac{8}{K(8-s)}\ee,\; \bm Y = - \frac{2}{K(8-s)}\left(\Ib + \frac{1}{K}\ee\ee^\top\right) $$
	\item Case 2: $2\leq s \leq 4$. Consider the following solution to the dual problem:
	$$\bm F_2 = \left(\frac{s}{4}\right)^2\ee\ee^\top ,\; \bm g_2 = \frac{s}{4}\ee,\; h_2 = 1,\; h_1 = 1+\frac{s}{4},\; \bm F_1 = \bm F_2 + \frac{s}{4}\ee\ee^\top,\; \bm g_1 = \bm g_2 - \frac{s}{4}\ee = \bm 0,$$
	and the following solution to the primal problem:
	$$\lambda_2 = 0,\; \lambda_1 = \frac{1}{4+s},\; \bm\rho_1 = \frac{2(s-2)}{K(4+s)} \begin{bmatrix}\bm 0\\ \ee \end{bmatrix},\; \bm\rho_2 = \bm 0,\; \bm y =  \frac{4+2s}{K(4+s)} \ee,\; \bm Y = - \frac{2}{K(4+s)}\left(\Ib + \frac{1}{K}\ee\ee^\top\right)$$
	\item Case 3: $s \geq 4$. Consider the following solution to the dual problem:
	$$\bm F_2 = \ee\ee^\top ,\; \bm g_2 = \ee,\; h_2 = 1,\; h_1 = 2,\; \bm F_1 = 2 \ee\ee^\top,\; \bm g_1 = \bm 0,$$
	and the following solution to the primal problem:
	$$\lambda_2 = 0,\; \lambda_1 = 0,\; \bm\rho_1 = \frac{1}{K}\begin{bmatrix}\bm 0\\ \ee \end{bmatrix},\; \bm\rho_2 = \frac{1}{K}\begin{bmatrix}\ee \\ \bm 0 \end{bmatrix},\; \bm y = \frac{1}{K}\ee,\; \bm Y = \bm 0.$$
	In all the three cases, the primal and the dual solutions are feasible to their respective problems, and provide the same objective value which corresponds to the one presented in Example \ref{example:dro2}.
\end{itemize}

\bibliographystyle{plain}
\bibliography{bibliography} 